\documentclass[11pt,reqno]{amsart}
\usepackage[colorlinks=true,linkcolor=red,citecolor=blue]{hyperref}
\usepackage{graphics}
\usepackage{tikz}
\usepackage{amssymb}
\usepackage{comment}
\usetikzlibrary{patterns}
\usepackage{enumerate}
\newcommand{\di}{\displaystyle}
\newcommand{\tron}[1]{\left( #1 \right)}
\usepackage{color}
\usepackage[left=3.00cm, right=3.00cm, top=2.00cm, bottom=2.00cm]{geometry}
\numberwithin{equation}{section}
\newtheorem{theorem}{Theorem}[section]
\newtheorem{remark}{Remark}[section]
\newtheorem{definition}{Definition}[section]
\newtheorem{lemma}[theorem]{Lemma}

\newtheorem{proposition}[theorem]{Proposition}

\newcommand{\flo}[1]{\lfloor #1 \rfloor }
\allowdisplaybreaks

\newcommand{\R}{\mathbb{R}}
\newcommand{\Z}{\mathbb{Z}}

\begin{document} 
\title[A system of semilinear damped $\sigma$-evolution equations]{On a system of semilinear damped $\sigma$-evolution equations with different damping types in the critical case}

\thanks{$^* $\textit{Corresponding author:} Tuan Anh Dao}

\maketitle
\centerline{\scshape \textbf{Trung Loc Tang}}
{\footnotesize
	\centerline{School of Mathematics and Computer Science, Hanoi National University of Education}
	\centerline{136 Xuan Thuy, Hanoi, Vietnam}
	\centerline{Email: ttloc.toank28cltt@gmail.com}}
\medskip

\centerline{\scshape \textbf{Tuan Anh Dao}}
{\footnotesize
	\centerline{Faculty of Mathematics and Informatics, Hanoi University of Science and Technology}
	\centerline{No.1 Dai Co Viet road, Hanoi, Vietnam}
	\centerline{Email: anh.daotuan@hust.edu.vn}}
\medskip

\centerline{\scshape \textbf{The Anh Cung}}
{\footnotesize
	\centerline{School of Mathematics and Computer Science, Hanoi National University of Education}
	\centerline{136 Xuan Thuy, Hanoi, Vietnam}
	\centerline{Email: anhctmat@hnue.edu.vn}}

\begin{abstract}
In this paper, we study the non-symmetric system of semilinear damped $\sigma$-evolution equations with different damping types, where two power exponents of nonlinearities belong to the critical curve, by using moduli of continuity in nonlinear terms. Our goal is to determine the sharp conditions on these moduli of continuity that guarantee the global (in time) existence of Sobolev solutions or, conversely, lead to finite-time blow-up. Furthermore, by employing the analysis introduced in the proof of our blow-up result, we provide a positive answer to an open problem for the symmetric models posed in the literature.
\end{abstract}

\tableofcontents

\section{Introduction}
To get started, let us make a brief overview about the research perspectives of the following Cauchy problem for structurally damped $\sigma$-evolution equations:
\begin{equation}\label{sys1}
\left\{
\begin{aligned}
&u_{tt}+(-\Delta)^{\sigma}u+(-\Delta)^{\delta}u_t=|u|^p, &&x\in \mathbb{R}^n,\, t\geq 0,\\
&(u,u_t)(0,x)= (u_0,u_1)(x), &&x\in \mathbb{R}^n,
\end{aligned}
\right.
\end{equation}
where $\sigma\ge 1$ and $\delta\in [0,\sigma]$, which have attracted significant attention from the mathematical community (see, for example, \cite{Mat1976,Duong2015,Dao2019,ADL2024} and references therein). The choice of $\delta$ in the damping term plays a crucial role in determining the asymptotic behavior of solutions to the corresponding linear equation of \eqref{sys1}. More precisely, in the case where $\delta \in [0, \sigma/2)$, the linear equation exhibits the ``parabolic-like" behavior, as its solutions share the same asymptotic profile as those of the so-called anomalous diffusion equation (see \cite{D'Abbicco2017}), which is exactly the heat equation when $\sigma=1$. Furthermore, the critical exponent for the semi-linear problem \eqref{sys1} under assuming additional $L^1$ regularity for the initial data is determined by $p_{\mathrm{par}}(n,\sigma,\delta) := 1 +2\sigma/(n-2\delta)$. On the other hand, when $\delta \in (\sigma/2,\sigma]$, the diffusion phenomenon no longer occurs to the linear equation, and the behavior of solutions coincides with that of the free $\sigma$-evolution equations, which is exactly the classical wave equation when $\sigma=1$. This regime is often referred to the ``$\sigma$-evolution-like" model. In this case, the critical exponent to \eqref{sys1}, under assuming additional $L^1$ regularity for the initial data, is given by $p_{\mathrm{evo}}(n,\sigma) := 1 + 2\sigma/(n-\sigma)$. Here, the critical exponent $p_{\rm crit}$ means that there exist global (in time) small data Sobolev solutions for some range of admissible exponent $p>p_{\rm crit}$, whereas one may find non-existence of global (in time) Sobolev solutions even for small data in the case $1< p < p_{\rm crit}$. In other words, we have only local (in time) Sobolev solutions in the latter case. From these observations, the value $\delta = \sigma/2$ acts as the threshold to distinguish the two distinct models appearing in terms of the study of \eqref{sys1} and its corresponding linear equation. A natural question arising is whether one can find the so-called critical curve for a weakly coupled system in the non-symmetric setting, i.e. a system consists of both classes of equations mentioned above. This problem, which focuses heavily on the interaction between a ``parabolic-like" equation and a ``$\sigma$-evolution-like" equation, has been recently addressed in \cite{DuongDaoRei2025} for the following system:
\begin{equation}\label{DuongDaoReisystem}
\left\{
\begin{aligned}
    &u_{tt}+(-\Delta)^{\sigma} u+(-\Delta)^{\sigma}u_t &&=|v|^{p}, &&x\in \mathbb{R}^n,\,t\geq 0,\\
    &v_{tt}+(-\Delta)^{\sigma}v+v_t &&=|u|^{q}, &&x\in \mathbb{R}^n,\,t\geq 0,\\
    &(u,u_t,v,v_t)(0,x) &&= (0,u_1,0,v_1)(x), &&x\in \mathbb{R}^n.
\end{aligned}
\right.
\end{equation}
The authors proved that under the conditions $1 < \sigma < n < 2\sigma$ and $p \geq 2$, if the inequality
$$ \max\left\{\frac{2q+1}{pq+q-2}, \frac{pq+p+1}{2pq-p-1}\right\} > \frac{n}{2\sigma} $$
holds, there exists a unique global (in time) Sobolev solution to \eqref{DuongDaoReisystem}. Additionally, they established a blow-up result if the relation
$$ \max\left\{\frac{2q+1}{pq+q-2}, \frac{pq+p+1}{2pq-p-1}\right\} < \frac{n}{2\sigma} $$
happens. Consequently, one realizes that the critical curve for \eqref{DuongDaoReisystem} is defined by
\begin{equation}\label{critcurve}\tag{$*$}
\max\left\{\frac{2q+1}{pq+q-2}, \frac{pq+p+1}{2pq-p-1}\right\}-\frac{n}{2\sigma}=0.
\end{equation}

Among other perspectives, the very interesting topic is to study a single equation or a weakly coupled system with power nonlinearities whose exponents take on the critical value or belong to the critical curve, respectively. One recognizes that this perspective has recently attracted a growing interest from numerous mathematicians over the world. The main advantage of considering such problems is to determine the critical nonlinearity or the critical regularity of nonlinearities by adding modulus of continuity in nonlinear terms. For this approach, we assume a function $\mu:[0,\infty)\to [0,\infty)$ to be continuous, concave, and increasing which provides an additional regularity of nonlinearities in comparison with the power nonlinear terms. We refer the reader to a few papers (see more \cite{EbeGirRei2020,AnhRei2021,ChenGirardi2025,TangDaoCung2026,TangDuong2026}) for recent results in this direction. One of the most remarkable contributions comes from Dao-Reissig in \cite{AnhRei2021}, where the authors have investigated the weakly coupled system of semilinear damped wave equations in the critical case as follows:
\begin{equation}\label{sys4}
\left\{
\begin{aligned}
&u_{tt}-\Delta u+u_t=|v|^{p^*}\mu_{1}(|v|), &&x\in \mathbb{R}^n,\, t\geq 0,\\
&v_{tt}-\Delta v+v_t=|u|^{q^*}\mu_{2}(|u|), &&x\in \mathbb{R}^n,\, t\geq 0,\\
&(u,u_t,v,v_t)(0,x)= (u_0,u_1,v_0,v_1)(x), &&x\in \mathbb{R}^n.
\end{aligned}
\right.
\end{equation}
Here $(p^*, q^*)$ belong to the critical curve described by 
\[
\frac{\max\{p,q\}+1}{pq-1}-\frac{n}{2}=0
\]
and $\mu_1=\mu_1(s),\, \mu_2=\mu_2(s)$ stand for moduli of continuity. The authors proved that \eqref{sys4} admits a global (in time) Sobolev solution if $I_{\mu_1,\mu_2} < \infty$, otherwise, the solution blows up in finite time when $I_{\mu_1,\mu_2} = \infty$, where the integral quantiy is defined by
\[
I_{\mu_1,\mu_2} :=
\begin{cases}
    \vspace{0.2cm}\displaystyle\int_{0}^{c} \frac{1}{s} (\mu_1(s))^{\frac{1}{1+p^*}} (\mu_2(s))^{\frac{p^*}{1+p^*}} \, ds &\text{ if } p^*\ge q^*, \\
    \displaystyle\int_{0}^{c} \frac{1}{s} (\mu_1(s))^{\frac{q^*}{1+q^*}} (\mu_2(s))^{\frac{1}{1+q^*}} \, ds &\text{ if } q^*\ge p^*, \\
\end{cases}
\]
for a suitably small constant $c>0$. We can say that this paper gives a new interaction of additional regularities of power nonlinearities in comparison with other research results in terms of the study of weakly coupled systems. As strongly motivated from \cite{DuongDaoRei2025,AnhRei2021}, in the present work we would like to investigate the following weakly coupled system:
\begin{equation}\label{mainsystem}
\begin{cases}
\begin{aligned}
    &u_{tt}+(-\Delta)^{\sigma} u+(-\Delta)^{\sigma}u_t &&=|v|^{p^*}\mu_1(|v|), &&x\in \mathbb{R}^n,\,t\geq 0,\\
    &v_{tt}+(-\Delta)^{\sigma} v+v_t &&=|u|^{q^*}\mu_2(|u|), &&x\in \mathbb{R}^n,\,t\geq 0,\\
    &(u,u_t,v,v_t)(0,x) &&= (0,u_1,0,v_1)(x), &&x\in \mathbb{R}^n,
\end{aligned}
\end{cases}
\end{equation}
where $(p^*,q^*)$ lies on the critical curve \eqref{critcurve}. {\textit{We aim to determine the sharp threshold for $\mu_1$ and $\mu_2$ that distinguishes between the global (in time) existence and the blow-up of Sobolev solutions to \eqref{mainsystem}}}. Obviously, we can see that several key features distinguish our considered system from those previously studied in the literature. First, the system \eqref{mainsystem} is non-symmetric, so that the roles of $u$ and $v$, as well as the power exponents $p^*$ and $q^*$, are not equivalent. This means that we cannot interchange their roles in our analysis (see later, \eqref{con3athr1.1}-\eqref{con3bthr1.1} in Theorem \ref{thr1.1} and \eqref{con3thr1.2}-\eqref{con4thr1.2} in Theorem \ref{thr1.2}) as appeared in previous works. Second, since the fractional Laplacian $(-\Delta)^{\sigma}$ becomes nonlocal, when $\sigma$ is fractional, standard compactly supported test functions are no longer applicable to establish the blow-up results. To overcome these difficulties, on the one hand the main strategy for the proof of global (in time) Sobolev solutions is based on using different weight functions of loss of decay in the solution space. On the other hand, we introduce a judiciously modified test function with non-compact support, which allows us to effectively estimate nonlocal operators in the proof of blow-up solutions. Finally, we also demonstrate that the test fucntion method developed in this paper can be successfully applied to provide a blow-up result to the symmetric models in Section \ref{subsec4.3}. \medskip

\textbf{Notations:} We write $f\lesssim g$ if there exists a constant $C>0$ such that $f\leq Cg$, and $f \sim g$ if $g\lesssim f\lesssim g$. As usual, the spaces $H^a$ and $\dot{H}^a$ with $a \geq 0$ stand for Bessel and Riesz potential spaces based on $L^2$ space. For a given number $s \in \R$, we denote $\flo{s}:= \max \big\{k \in \Z \,\, : \,\, k\leq s \big\}$ as its integer part and $[s]^{+}:=\max\{s,0\}$. We put $|x|:= \sqrt{x_1^2+x_2^2+\cdots+ x_n^2}$, the norm of $x \in \R^n$. Finally, we introduce the initial data space $(u_1,v_1)\in \mathcal{D}:=\tron{L^1\cap L^2}^2$ with the corresponding norm
\begin{align*}          
\|(u_1,v_1)\|_{\mathcal{D}}:=\|u_1\|_{L^2}+\|u_1\|_{L^1}+\|v_1\|_{L^2}+\|v_1\|_{L^1}.
\end{align*}

\textbf{Main results:} Let us introduce the the following two quantities:
$$ p_{\rm c}:= p_{\mathrm{par}}(n,\sigma,0)= 1 + 2\sigma/n \text{ and }q_{\rm c}:= p_{\mathrm{evo}}(n,\sigma)= 1 + 2\sigma/(n-\sigma). $$
Since $(p^*, q^*)$ satisfies \eqref{critcurve}, it follows that
\[
p^* \leq p_{\rm c} \text{ and } q^* \geq q_{\rm c}, \quad \text{or} \quad p^* \geq p_{\rm c} \text{ and } q^* \leq q_{\rm c}.
\]
The first result concerns with the global (in time) existence of Sobolev solutions, for which the assumption $p^*\geq 2$ is required (see \cite{DuongDaoRei2025}).
\begin{theorem}\label{thr1.1}
    Let $1<\sigma<n<2\sigma$ and the moduli of continuity satisfy the following assumption:
     \begin{equation}\label{con1thr1.1}
        s\mu'_{j}(s)=\mathcal{O}(\mu_j(s))\quad \text{as $s\to +0$ with $j=1,2$.}
    \end{equation}
    Moreover, we suppose that one of the following conditions is satisfied:
    \begin{itemize}
        \item[\rm i)] 
        \begin{equation}\label{con2thr1.1}
            \di\int_{0}^c\frac{\mu_1(s)}{s}\;ds<\infty \text{ and }\di\int_{0}^c\frac{\mu_2(s)}{s}\;ds<\infty,
        \end{equation}
        \item[\rm ii)] 
            \begin{equation}
            \di\int_{0}^c\frac{1}{s}\tron{\mu_1(s)}^{\frac{q^*}{q^*+1}}\tron{\mu_2(s)}^{\frac{1}{q^*+1}}\;ds<\infty \text{ if }\eqref{con2thr1.1} \text{ fails and }q^*\geq q_{\rm c},\,2\leq p^*\leq p_{\rm c}, \label{con3athr1.1}
            \end{equation}
        \item[\rm iii)] 
            \begin{equation}
            \di\int_{0}^c\frac{1}{s}\tron{\mu_1(s)}^{\frac{1}{p^*+1}}\tron{\mu_2(s)}^{\frac{p^*}{p^*+1}}\;ds<\infty\text{ if }\eqref{con2thr1.1} \text{ fails and }q^*< q_{\rm c},\, p^*>p_{\rm c}.\label{con3bthr1.1}
            \end{equation}    
    \end{itemize}
    Especially, when $p^*=p_{\rm c},q^*=q_{\rm c}$, a further assumption required is that $s\in (0,c]\to \mu_1(s)/\mu_2(s)$ is a decreasing function, where $c>0$ is a suitable small constant. Then, there exists a constant $\varepsilon$ such that for any small data $\left(u_1, v_1\right) \in \mathcal{D}$ satisfying the assumption $\|(u_1,v_1)\|_{\mathcal{D}}\leq\varepsilon$, we have a uniquely determined global (in time) Sobolev solution
$$
(u, v) \in \mathcal{C}\big([0, \infty),\dot{H}^{\sigma}\cap L^{q^*}\cap L^{\infty}\big)\cap \mathcal{C}\left([0,\infty),H^{\sigma}\right)
$$
to \eqref{mainsystem}. Moreover, the following estimates hold for $t>0$:
$$
\begin{aligned}
\|u(t,\cdot)\|_{L^{q^*}} &\lesssim (1+t)^{1-\frac{n}{\sigma}\tron{1-\frac{1}{q^*}}+[\varepsilon_1(p^*)]^{+}}\ell_1(t)\|(u_1,v_1)\|_{\mathcal{D}},\\
\|u(t,\cdot)\|_{L^{\infty}} &\lesssim (1+t)^{1-\frac{n}{\sigma}+[\varepsilon_1(p^*)]^{+}}\ell_1(t)\|(u_1,v_1)\|_{\mathcal{D}},\\
\|u(t, \cdot)\|_{\dot{H}^{\sigma}} &\lesssim (1+t)^{-\frac{n}{4\sigma}+[\varepsilon_1(p^*)]^{+}}\ell_1(t)\|(u_1,v_1)\|_{\mathcal{D}},\\
\|v(t, \cdot)\|_{L^2} &\lesssim (1+t)^{-\frac{n}{4\sigma}+[\varepsilon_2(q^*)]^{+}}\ell_2(t)\|(u_1,v_1)\|_{\mathcal{D}},\\
\|v(t,\cdot)\|_{L^{\infty}} &\lesssim (1+t)^{-\frac{n}{2\sigma}+[\varepsilon_2(q^*)]^{+}}\ell_2(t)\|(u_1,v_1)\|_{\mathcal{D}},\\
\|v(t,\cdot)\|_{\dot{H}^{\sigma}} &\lesssim (1+t)^{-\frac{n}{4\sigma}-\frac{1}{2}+[\varepsilon_2(q^*)]^{+}}\ell_2(t)\|(u_1,v_1)\|_{\mathcal{D}},
\end{aligned}
$$
where
\begin{equation}\label{con4thr1.1}
\begin{cases}
    \varepsilon_1(p^*):=1-\frac{n}{2\sigma}(p^*-1),\\
    \varepsilon_2(q^*):=1+q^*-\frac{n}{\sigma}(q^*-1),
\end{cases}
\end{equation}
the weight functions $\ell_1(t)$ and $ \ell_2(t)$ are defined by
\begin{equation}\label{con5thr1.1}
\ell_1(t):= 
\begin{cases}
1 &\text { if \eqref{con2thr1.1} or \eqref{con3bthr1.1} holds}, \\
\left(\dfrac{\mu_1\left(c(1+t)^{-\gamma}\right)}{\mu_2\left(c(1+t)^{-\gamma}\right)}\right)^{\frac{1}{q^*+1}}  &\text { if \eqref{con3athr1.1} holds},
\end{cases}
\end{equation}
and
\begin{equation}\label{con6thr1.1}
\ell_2(t):=
\begin{cases}
1 &\text { if \eqref{con2thr1.1} or \eqref{con3athr1.1} holds}, \\
\left(\dfrac{\mu_2\left(c(1+t)^{-\gamma}\right)}{\mu_1\left(c(1+t)^{-\gamma}\right)}\right)^{\frac{1}{p^*+1}}  &\text { if \eqref{con3bthr1.1} holds},
\end{cases}
\end{equation}
respectively, with a sufficiently small constant $\gamma>0$. 
\end{theorem}
\begin{remark}
{\rm
   The terms $\varepsilon_1(p^*),\, \varepsilon_2(q^*)$ and the functions $\ell_1(t),\,\ell_2(t)$ present some loss of decay for solutions to the nonlinear problem \eqref{mainsystem} in comparison with the corresponding linear problems \eqref{eqlinearhyp}, \eqref{eqlinearpar}. In the special case $p^*=p_{\rm c},\,q^*=q_{\rm c}$, if $\mu_1(s)/\mu_2(s)$ is an increasing function, then the Sobolev solutions exist globally, provided that the condition \eqref{con3athr1.1} is replaced by the condition \eqref{con3bthr1.1}.
}
\end{remark}
In the case where the conditions \eqref{con3athr1.1} and \eqref{con3bthr1.1} are violated, we demonstrate that the solution blows up in finite time, that is, the problem \eqref{mainsystem} admits no global (in time) Sobolev solution. For this purpose, let us introduce
\begin{equation}\label{critnonlinearity}
    \Phi_p:=\Phi_p(s)=s^{p^*} \mu_1(s) \text { and } \Phi_q:=\Phi_q(s)=s^{q^*} \mu_2(s).
\end{equation}
\begin{definition}\label{def_weaksolution}
 The pair of functions $(u,v)$ is called a global (in time) Sobolev solution to \eqref{mainsystem} if 
 \[
 \Phi_p(|v(t,x)|),\, \Phi_q(|u(t,x)|)\in L^1\big((0,\infty)\times \R^n\big)
 \]
 and the following identities hold:
    \[
    \begin{aligned}
    &\int_{0}^{\infty}\int_{\R^n}u(t,x)\left[\partial^2_t-\partial_t(-\Delta)^{\sigma}+(-\Delta)^{\sigma}\right]\phi_1(t,x)dxdt-\int_{\mathbb{R}^n}u_1(x)\phi_1(0,x)dx\\
    &\qquad=\int_{0}^{\infty}\int_{\R^n}\Phi_{p}(|v(t,x)|)\phi_1(t,x)dxdt, \\
    &\int_{0}^{\infty}\int_{\R^n}v(t,x)\left[\partial^2_t-\partial_t+(-\Delta)^{\sigma}\right]\phi_2(t,x)dxdt-\int_{\mathbb{R}^n}v_1(x)\phi_2(0,x)dx\\
    &\qquad =\int_{0}^{\infty}\int_{\R^n}\Phi_q(|u(t,x)|)\phi_2(t,x)dxdt 
\end{aligned}
    \]
    for any test function $\phi_{i}(t,x)$ of the form $\phi_{i}(t,x)=\eta_i(t)\varphi_{i}(x)$, where $\eta_i(t) \in \mathcal{C}^{\infty}_0\big([0,\infty)\big)$, $\eta_i\equiv 1$ in a neighborhood of $0$, and $\varphi_{i}(x)\in H^{2\sigma}$ with $i=1,2$.
\end{definition}
\begin{remark}
{\rm
     Thanks to the embedding result $H^{2\sigma}\hookrightarrow L^{\infty}$ for $n<2\sigma$, it follows that
    \[
    \int_{0}^{\infty}\int_{\R^n}\Phi_p(|v(t,x)|)\phi_1(t,x)dxdt \lesssim \|\varphi_1\|_{L^{\infty}}\left\|\Phi_p\right\|_{L^1((0,\infty)\times \R^n)}<+\infty,
    \]
    similarly,
    \[
    \int_{0}^{\infty}\int_{\R^n}\Phi_q(|u(t,x)|)\phi_2(t,x)dxdt \lesssim \|\varphi_2\|_{L^{\infty}}\left\|\Phi_q\right\|_{L^1((0,\infty)\times \R^n)}<+\infty,
    \]
    which show that Definition \ref{def_weaksolution} makes sense.
}
\end{remark}
\begin{theorem}\label{thr1.2}
    Let $1<\sigma<n<2\sigma$. We assume that the initial data $u_1,v_1\in L^1 $ satisfy
    \begin{equation}\label{con1thr1.2}
        \int_{\mathbb{R}^n}u_1(x)dx>0\quad \text{and}\quad \int_{\mathbb{R}^n}v_1(x)dx>0.
    \end{equation}
    Moreover, we suppose the following assumptions of the moduli of continuity
    \begin{equation}\label{con2thr1.2}
        s^{k}\mu^{(k)}_{j}(s)=\mathcal{O}(\mu_j(s))\text{ as }s\to +0\text{ with }j,k=1,2,
    \end{equation}
    together with
    \begin{itemize}
         \item[\rm i)] 
     \begin{equation}\label{con3thr1.2}
            \di\int_{0}^c\frac{1}{s}\tron{\mu_1(s)}^{\frac{q^*}{q^*+1}}\tron{\mu_2(s)}^{\frac{1}{q^*+1}}\;ds=\infty \text{ if }q^*\geq q_{\rm c},\, p^*\leq p_{\rm c},
            \end{equation}
         \item[\rm ii)]     
            \begin{equation}\label{con4thr1.2}
            \di\int_{0}^c\frac{1}{s}\tron{\mu_1(s)}^{\frac{1}{p^*+1}}\tron{\mu_2(s)}^{\frac{p^*}{p^*+1}}\;ds=\infty\text{ if }q^*\leq q_{\rm c},\, p^*\geq p_{\rm c},
            \end{equation}
    \end{itemize}
    where $c>0$ is a suitable small constant and $(p^*,q^*)$ belongs to the curve \eqref{critcurve}. Especially, when $\sigma$ is not an integer and \eqref{con4thr1.2} holds, we require the additional assumption that $n\geq 2\flo{\sigma}$. Then, there is no global (in time) Sobolev solution $(u,v)$ to \eqref{mainsystem} in sense of Definition \ref{def_weaksolution}.
\end{theorem}
\begin{remark}
{\rm
Theorem \ref{thr1.2} tells us that our result gives a non-trivial improvement over the existing literature. The novelty of this work comes from handling the fractional Laplacian operators by setting them in the weakly coupled system with the critical case.
}
\end{remark}
\begin{remark}
\rm{
    We refer interested readers to Example 1.1 and Example 1.2 in \cite{AnhRei2021} for admissible choices of $\mu_1, \mu_2$ satisfying \eqref{con2thr1.1} or \eqref{con3athr1.1} or \eqref{con3bthr1.1} in Theorem \ref{thr1.1}, and \eqref{con3thr1.2} or \eqref{con4thr1.2} in Theorem \ref{thr1.2}, respectively. 
}
\end{remark}
\section{Global existence of Sobolev solutions}
In this section, we give the proof of Theorem \ref{thr1.1}. We denote by $K_0=K_0(t,x)$ and $K_1=K_1(t,x)$ the fundamental solutions to the linear Cauchy problems \eqref{eqlinearhyp} and \eqref{eqlinearpar}, respectively. Then, the solutions to \eqref{mainsystem} with vanishing right-hand sides are defined by
$$
\left\{\begin{array}{l}
u^{\mathrm{lin}}(t,x)=K_0(t,x) *_x u_1(x), \\
v^{\mathrm{lin} }(t,x)=K_1(t,x) *_x v_1(x) .
\end{array}\right.
$$
Thanks to Duhamel's principle, the solutions to \eqref{mainsystem} are written by the following system of nonlinear integral equations:
\[
\left\{
\begin{aligned}
u(t,x)=u^{\mathrm{lin}}(t,x)+\int_0^t K_0(t-\tau,x) *_x\left(|v(\tau,x)|^{p^*} \mu_1(|v(\tau,x)|)\right) d \tau=: u^{\mathrm{lin}}(t,x)+u^{\mathrm{nl}}(t,x), \\
v(t,x)=v^{\mathrm{lin}}(t,x)+\int_0^t K_1(t-\tau,x) *_x\left(|u(\tau,x)|^{q^*} \mu_2(|u(\tau,x)|)\right) d \tau=: v^{\mathrm{lin} }(t,x)+v^{\mathrm{nl}}(t,x) .
\end{aligned}
\right.
\]
For all $t>0$, we define the operator in a solution space $X(t)$ as follows:
$$
\Psi: \quad(u, v) \in X(t) \mapsto \Psi(u, v)(t,x)=\left(u^{\mathrm{lin} }(t,x)+u^{\mathrm{nl}}(t,x), v^{\mathrm{lin} }(t,x)+v^{\mathrm{nl}}(t,x)\right) .
$$
Our aim is to apply Banach's fixed point theorem in the proof of global (in time) existence of small data Sobolev solutions to \eqref{mainsystem}. For this purpose, we need to verify the following two inequalities:
\begin{equation}\label{eqbanachfixedpoint1}
    \|\Psi(u, v)\|_{X(t)} \lesssim\left\|\left(u_1,v_1\right)\right\|_{\mathcal{D}}+\|(u, v)\|_{X(t)}^{p^*}+\|(u, v)\|_{X(t)}^{q^*},
\end{equation}
\begin{equation}\label{eqbanachfixedpoint2}
\begin{aligned}
\|\Psi(u, v)-\Psi(\tilde{u}, \tilde{v})\|_{X(t)} &\lesssim\|(u, v)-(\tilde{u}, \tilde{v})\|_{X(t)}\left(\|(u, v)\|_{X(t)}^{p^*-1}+\|(\tilde{u}, \tilde{v})\|_{X(t)}^{p^*-1}\right. \\
&\qquad\left.+\|(u, v)\|_{X(t)}^{q^*-1}+\|(\tilde{u}, \tilde{v})\|_{X(t)}^{q^*-1}\right).
\end{aligned}
\end{equation}
 
\subsection{Key estimates and inequalities} To establish \eqref{eqbanachfixedpoint1} and \eqref{eqbanachfixedpoint2}, we recall some useful estimates for solutions to the linear problems as well as fundamental inequalities, which will be employed in our proofs. Namely, the two corresponding linear problems of \eqref{mainsystem} we have in mind are
\begin{equation}\label{eqlinearhyp}
\begin{cases}
    w_{tt}(t,x)+(-\Delta)^{\sigma}w(t,x)+(-\Delta)^{\sigma}w_t(t,x)=0, &x\in \mathbb{R}^n,\,t\geq 0,\\
    (w,w_t)(0,x)=(0,w_1)(x), &x\in \mathbb{R}^n,
\end{cases}
\end{equation}
and 
\begin{equation}\label{eqlinearpar}
\begin{cases}
    w_{tt}(t,x)+(-\Delta)^{\sigma}w(t,x)+w_t(t,x)=0, &x\in \mathbb{R}^n,\,t\geq 0,\\
    (w,w_t)(0,x)=(0,w_1)(x), &x\in \mathbb{R}^n,
\end{cases}
\end{equation}
\begin{proposition}[see \cite{DAbEbe2021}] \label{Linear estimates hyp} Let $\sigma>0$ and $\sigma\ne 1$. Assume that $m \in [1,2]$ and $w_1 \in L^m \cap L^2$. Then, the following estimate for Sobolev solutions to \eqref{eqlinearhyp} holds:
\begin{align*}
    \|\big(|D|^{\sigma},\partial_t\big)w(t,\cdot)\|_{L^2} &\lesssim (1+t)^{-\frac{n}{2\sigma}(\frac{1}{m}-\frac{1}{2})} \|w_1\|_{L^m \cap L^2}.
\end{align*}
Furthermore, if the following conditions are provided for $1 \leq \alpha_1 \leq m \leq \alpha_2 \leq \infty$:
$$ n\left(\frac{1}{\alpha_1}-\frac{1}{\alpha_2}\right)+n\sigma \max\left\{\frac{1}{2}-\frac{1}{\alpha_1}, \frac{1}{\alpha_2}-\frac{1}{2}\right\} < \sigma $$
and
$$ \frac{1}{2} \leq \frac{1}{m}-\frac{1}{\alpha_2} < \frac{2\sigma}{n}, $$
then we conclude that
\begin{align*}
    \|w(t,\cdot)\|_{L^{\alpha_2}} \lesssim (1+t)^{-\frac{n}{\sigma}(\frac{1}{\alpha_1}-\frac{1}{\alpha_2})+1}\|w_1\|_{L^{\alpha_1}} + e^{-ct} \|w_1\|_{L^m},
\end{align*}
where $c$ is a suitable positive constant.
\end{proposition}

\begin{proposition}[see \cite{PhamRei2017,ChenGirardi2025}] \label{Linear estimate par}
   Let $\sigma\geq 1$ and $n<2\sigma$. Assume that $m \in [1,2]$ and $w_1 \in L^m \cap L^2$. Then, the following estimates for Sobolev solutions to \eqref{eqlinearpar} hold:
   \begin{align*}
       \|w(t,\cdot)\|_{L^2}&\lesssim (1+t)^{-\frac{n}{4\sigma}}\|w_1\|_{L^1\cap L^2},\\
       \||D|^{\sigma}w(t,\cdot)\|_{L^2}&\lesssim (1+t)^{-\frac{n}{4\sigma}-\frac{1}{2}}\|w_1\|_{L^1\cap L^2},\\
       \|w(t,\cdot)\|_{L^\infty}&\lesssim (1+t)^{-\frac{n}{2\sigma}}\|w_1\|_{L^1\cap L^2}.      
   \end{align*}
\end{proposition}

\begin{proposition}[Fractional Gagliardo-Nirenberg inequality] \label{G-N}
Let $1<r, r_0, r_1<\infty, \rho>0$ and $s \in[0, 1)$. Then, it holds:
$$
\|w\|_{\dot{H}_r^s} \lesssim\|w\|_{L^{r_0}}^{1-\theta}\|w\|_{\dot{H}_{r_1}^\rho}^\theta,
$$
where $\theta=\theta_{s,\rho}\left(r, r_0, r_1\right)=\dfrac{\frac{1}{r_0}-\frac{1}{r}+\frac{s}{n}}{\frac{1}{r_0}-\frac{1}{r_1}+\frac{\rho}{n}}$ and $\dfrac{s}{\rho} \leq \theta \leq 1$.
\end{proposition}

\begin{lemma}[see \cite{AnhRei2021}] \label{integralmoduluslemma} 
Let $\mu_1=\mu_1(s)$ and $\mu_2=\mu_2(s)$ be moduli of continuity. Then, the following estimates hold:
\begin{align*}
\di &\int_0^t(1+t-\tau)^{-\alpha_1}(1+\tau)^{-1}\left(\mu_1\left(C(1+\tau)^{-\alpha_2}\right)\right)^{\beta_1}\left(\mu_2\left(C(1+\tau)^{-\alpha_2}\right)\right)^{\beta_2} d \tau\\
&\qquad\lesssim(1+t)^{-\alpha_1} \int_0^t(1+\tau)^{-1}\left(\mu_1\left(C(1+\tau)^{-\alpha_2}\right)\right)^{\beta_1}\left(\mu_2\left(C(1+\tau)^{-\alpha_2}\right)\right)^{\beta_2} d \tau
\end{align*}
for any $\alpha_1 \leq 1$ and for all $\alpha_2 ,\beta_1,\beta_2\geq 0$.
\end{lemma}

\begin{lemma}\label{ell1ell2}
    Let $k>0$ and $\ell_1(t),\ell_2(t)$ are defined as in the statement of Theorem \ref{thr1.1}. Then, the following properties hold:
    \begin{itemize}
        \item[\rm i.] $(1+t)^{k}\ell_1(t)$ and $(1+t)^{k}\ell_2(t)$ are increasing functions for $t>0$,
        \item[\rm ii.] $(1+t)^{-k}\ell_1(t)\leq (1+t)^{-\gamma}$ and $(1+t)^{-k}\ell_2(t)\leq (1+t)^{-\gamma}$ for some constant $\gamma\in (0,k)$.
    \end{itemize}

\end{lemma}
\begin{proof}
    The proof of this lemma follows from direct calculations together with the condition \eqref{con1thr1.1}, hence, we may omit its details.
\end{proof}

\subsection{Proof of Theorem 1.1}
We introduce the solution space
$$
X(t):= \mathcal{C}([0, t],\dot{H}^{\sigma}\cap L^{q^*}\cap L^{\infty})\times \mathcal{C}([0,t],H^{\sigma})
$$
with the norm
$$
\begin{aligned}
&\|(u, v)\|_{X(t)}:=\sup _{0 \leq \tau \leq t}\left((1+\tau)^{\frac{n}{\sigma}\tron{1-\frac{1}{q^*}}-1-[\varepsilon_2(q^*)]^{+}}\big(\ell_2(\tau)\big)^{-1}\|u(\tau, \cdot)\|_{L^{q^*}}\right.\\
&\qquad\left.+(1+\tau)^{\frac{n}{\sigma}-1-[\varepsilon_2(q^*)]^{+}}\big(\ell_2(\tau)\big)^{-1}\|u(\tau, \cdot)\|_{L^\infty}+(1+\tau)^{\frac{n}{4\sigma}-[\varepsilon_2(q^*)]^{+}}\big(\ell_2(\tau)\big)^{-1}\|u(\tau, \cdot)\|_{\dot{H}^\sigma}\right.\\
&\qquad\left.+(1+\tau)^{\frac{n}{4\sigma}-[\varepsilon_1(p^*)]^{+}}\big(\ell_1(\tau)\big)^{-1}\|v(\tau, \cdot)\|_{L^2}+(1+\tau)^{\frac{n}{4\sigma}+\frac{1}{2}-[\varepsilon_1(p^*)]^{+}}\big(\ell_1(\tau)\big)^{-1}\|v(\tau, \cdot)\|_{\dot{H}^\sigma}\right.\\
&\qquad\left.+(1+\tau)^{\frac{n}{2\sigma}-[\varepsilon_1(p^*)]^{+}}\big(\ell_1(\tau)\big)^{-1}\|v(\tau,\cdot)\|_{L^{\infty}}\right),
\end{aligned}
$$
where the parameter $\varepsilon_1(p^*),\varepsilon_2(q^*)$ and the weight functions $\ell_1(\tau), \ell_2(\tau)$ are determined as in \eqref{con4thr1.1}, \eqref{con5thr1.1} and \eqref{con6thr1.1}. From Propositions \ref{Linear estimates hyp} and \ref{Linear estimate par}, it suffices to prove the following inequality instead of \eqref{eqbanachfixedpoint1}:
\begin{equation}\label{eq1thr1.1}
\big\|\big(u^{\mathrm{nl}}, v^{\mathrm{nl}}\big)\big\|_{X(t)} \lesssim\|(u, v)\|_{X(t)}^{p^*}+\|(u, v)\|_{X(t)}^{q^*} .
\end{equation}
For $k\in\{q^*,\infty\}$, using Proposition \ref{Linear estimates hyp} we get
\begin{align}
\big\|u^{\mathrm{nl}}(t,\cdot)\big\|_{L^k}&\lesssim \int_0^t(1+t-\tau)^{-\frac{n}{\sigma}\left(1-\frac{1}{k}\right)+1}\big\||v(\tau, \cdot)|^{p^*} \mu_1\big(|v(\tau, \cdot)|\big)\big\|_{L^1\cap L^2} d \tau,\label{eq2thr1.1}\\
\big\|u^{\mathrm{nl}}(t,\cdot)\big\|_{\dot{H}^{\sigma}} &\lesssim \int_0^t(1+t-\tau)^{-\frac{n}{4\sigma}}\big\||v(\tau, \cdot)|^{p^*} \mu_1\big(|v(\tau, \cdot)|\big)\big\|_{L^1\cap L^2} d \tau.\label{eq3thr1.1}
\end{align}
For $k\in\{2,\infty\}$, applying Proposition \ref{Linear estimates hyp} again one obtains
\begin{align}
\big\|v^{\mathrm{nl}}(t,\cdot)\big\|_{L^k} &\lesssim \int_0^t(1+t-\tau)^{-\frac{n}{2\sigma}\tron{1-\frac{1}{k}}}\big\||u(\tau, \cdot)|^{q^*} \mu_2\big(|u(\tau, \cdot)|\big)\big\|_{L^1\cap L^2} d \tau,\label{eq4thr1.1}\\ 
\big\|v^{\mathrm{nl}}(t,\cdot)\big\|_{\dot{H}^{\sigma}} &\lesssim \int_0^t(1+t-\tau)^{-\frac{n}{4\sigma}-\frac{1}{2}}\big\||u(\tau, \cdot)|^{q^*} \mu_2\big(|u(\tau, \cdot)|\big)\big\|_{L^1\cap L^2} d \tau.\label{eq5thr1.1}
\end{align}
In order to estimate $|v(\tau, \cdot)|^{p^*} \mu_1\big(|v(\tau, \cdot)|\big)$ and $|u(\tau, \cdot)|^{q^*} \mu_2\big(|u(\tau, \cdot)|\big)$ in $L^1 \cap L^2$-norm, we proceed as follows:
\begin{align*}
\big\||v(\tau, \cdot)|^{p^*} \mu_1\big(|v(\tau, \cdot)|\big)\big\|_{L^1\cap L^2} \lesssim \big\||v(\tau, \cdot)|^{p^*}\big\|_{L^1\cap L^2}\big\|\mu_1\big(|v(\tau, \cdot)|\big)\big\|_{L^{\infty}},\\
\big\||u(\tau, \cdot)|^{q^*} \mu_2\big(|u(\tau, \cdot)|\big)\big\|_{L^1\cap L^2} \lesssim\big\||u(\tau, \cdot)|^{q^*}\big\|_{L^1\cap L^2}\big\|\mu_2\big(|u(\tau, \cdot)|\big)\big\|_{L^{\infty}}.
\end{align*}
Since $\mu_1$ is an increasing function, we have
$$
\begin{aligned}
\big\|\mu_1\big(|v(\tau, \cdot)|\big)\big\|_{L^{\infty}} \leq \mu_1\big(\|v(\tau, \cdot)\|_{L^{\infty}}\big) \leq \mu_1\left(C_1(1+\tau)^{-\frac{n}{2\sigma}+[\varepsilon_1(p^*)]^{+}}\ell_1(\tau)\|(u,v)\|_{X(t)}\right).
\end{aligned}
$$
By taking $\|(u,v)\|_{X(t)} \leq \varepsilon_0$, where $\varepsilon_0$ is a sufficiently small constant, it follows that
$$
\big\|\mu_1\big(|v(\tau, \cdot)|\big)\big\|_{L^{\infty}} \leq \mu_1\left(C_1 \varepsilon_0(1+\tau)^{-\frac{n}{2\sigma}+[\varepsilon_1(p^*)]^{+}}\ell_1(\tau)\right),
$$
and similarly,
\[
\big\|\mu_2\big(|u(\tau, \cdot)|\big)\big\|_{L^{\infty}} \leq\mu_2\left(C_1 \varepsilon_0(1+\tau)^{-\frac{n}{\sigma}+1+[\varepsilon_2(q^*)]^{+}}\ell_2(\tau)\right).
\]
On the other hand, one finds
$$
\begin{aligned}
\big\||v(\tau, \cdot)|^{p^*}\big\|_{L^1 \cap L^2}=\big\||v(\tau, \cdot)|^{p^*}\big\|_{L^1}+\big\||v(\tau, \cdot)|^{p^*}\big\|_{L^2}=\big\|v(\tau, \cdot)\big\|_{L^{p^*}}^{p^*}+\big\|v(\tau, \cdot)\big\|_{L^{2 p^*}}^{p^*},\\
\big\||u(\tau, \cdot)|^{q^*}\big\|_{L^1 \cap L^2}=\big\||u(\tau, \cdot)|^{q^*}\big\|_{L^1}+\big\||u(\tau, \cdot)|^{q^*}\big\|_{L^2}=\big\|u(\tau, \cdot)\big\|_{L^{q^*}}^{q^*}+\big\|u(\tau, \cdot)\big\|_{L^{2 q^*}}^{q^*}.
\end{aligned}
$$
Employing the fractional Gagliardo-Nirenberg inequality from Proposition \ref{G-N} gives
$$
\begin{aligned}
\big\|v(\tau, \cdot)\big\|_{L^{p^*}}^{p^*} & \lesssim (1+\tau)^{-\frac{n}{2 \sigma}\left(p^*-1\right)+[\varepsilon_2(q^*)]^{+}p^*}\big(\ell_2(\tau)\big)^{p^*}\|(u,v)\|_{X(t)}^{p^*}, \\
\big\|v(\tau, \cdot)\big\|_{L^{2 p^*}}^{p^*} & \lesssim(1+\tau)^{-\frac{n}{2\sigma}\left(p^*-\frac{1}{2}\right)+[\varepsilon_2(q^*)]^{+}p^*}\big(\ell_2(\tau)\big)^{p^*}\|(u,v)\|_{X(t)}^{p^*},\\
\big\|u(\tau, \cdot)\big\|_{L^{q^*}}^{q^*} & \lesssim(1+\tau)^{q^*-\frac{n}{\sigma}\left(q^*-1\right)+[\varepsilon_1(p^*)]^+q^*}\big(\ell_1(\tau)\big)^{q^*}\|(u,v)\|_{X(t)}^{q^*}, \\
\big\|u(\tau, \cdot)\big\|_{L^{2q^*}}^{q^*} & \lesssim(1+\tau)^{q^*-\frac{n}{\sigma}\left(q^*-\frac{1}{2}\right)+[\varepsilon_1(p^*)]^{+}q^*}\big(\ell_1(\tau)\big)^{q^*}\|(u,v)\|_{X(t)}^{q^*}.
\end{aligned}
$$
As a result, we derive the following estimates:
\begin{align}
&\big\||v(\tau, \cdot)|^{p^*} \mu_1\big(|v(\tau, \cdot)|\big)\big\|_{L^1\cap L^2}\nonumber\\
&\quad\lesssim (1+\tau)^{-\frac{n}{2 \sigma}\left(p^*-1\right)+[\varepsilon_2(q^*)]^{+}p^*}\big(\ell_2(\tau)\big)^{p^*}\mu_1\tron{c(1+\tau)^{-\frac{n}{2\sigma}+[\varepsilon_1(p^*)]^{+}}\ell_1(\tau)}\|(u,v)\|^{p^*}_{X(t)},\label{eq6thr1.1}\\
&\big\||u(\tau, \cdot)|^{q^*} \mu_2\big(|u(\tau, \cdot)|\big)\big\|_{L^1 \cap L^2}\nonumber\\
&\lesssim (1+\tau)^{q^*-\frac{n}{\sigma}\left(q^*-1\right)+[\varepsilon_1(p^*)]^+q^*}\big(\ell_1(\tau)\big)^{q^*}\mu_2\tron{c(1+\tau)^{-\frac{n}{\sigma}+1+[\varepsilon_2(q^*)]^{+}}\ell_2(\tau)}\|(u,v)\|^{q^*}_{X(t)}.\label{eq7thr1.1}
\end{align}
{\bf At first, let us consider the case $p^*\leq p_{\rm c}\text{ and }q^*\geq q_{\rm c}$}. In this case, the critical curve \eqref{critcurve} becomes
\begin{equation}\label{critcurvecase1}
    \frac{2q^*+1}{p^*q^*+q^*-2}=\frac{n}{2\sigma}.
\end{equation}
Assuming \eqref{con2thr1.1}, we set $\ell_1(\tau) \equiv \ell_2(\tau) \equiv 1$, $\varepsilon_1(p^*)= 1 - (n/2\sigma)(p^*-1)$ and $\varepsilon_2(q^*) \equiv 0$. Consequently, the following relations hold:
\[
\begin{aligned}
-\frac{n}{2\sigma}\left(p^*-1\right)=\varepsilon_1(p^*)-1,\quad q^*-\frac{n}{\sigma}\left(q^*-1\right)+q^*\varepsilon_1(p^*)=-1.
\end{aligned}
\]
After applying the above equalities to \eqref{eq6thr1.1} and \eqref{eq7thr1.1}, we arrive at
\begin{align}
\big\||v(\tau, \cdot)|^{p^*} \mu_1\big(|v(\tau, \cdot)|\big)\big\|_{L^1\cap L^2}&\lesssim (1+\tau)^{-1+\varepsilon_1(p^*)}\mu_1\tron{c(1+\tau)^{-\frac{n}{2\sigma}}}\|(u,v)\|^{p^*}_{X(t)},\label{eq8thr1.1}\\
\big\||u(\tau, \cdot)|^{q^*} \mu_2\big(|u(\tau, \cdot)|\big)\big\|_{L^1 \cap L^2}&\lesssim (1+\tau)^{-1}\mu_2\tron{c(1+\tau)^{-\frac{n}{\sigma}+1+\varepsilon_1(p^*)}}\|(u,v)\|^{q^*}_{X(t)}\label{eq9thr1.1}.
\end{align}
Based on these observations together with \eqref{eq3thr1.1} and \eqref{eq8thr1.1}, we obtain the following chain of inequalities:
\begin{align}
\big\||D|^{\sigma} u^{\mathrm{nl}}(t, \cdot)\big\|_{L^2} 
&\lesssim (1+t)^{-\frac{n}{4\sigma}+\varepsilon_1(p^*)} \|(u, v)\|_{X(t)}^{p^*} \int_0^{t} (1+\tau)^{-1} \mu_1\left(c(1+\tau)^{-\frac{n}{2\sigma}}\right) d \tau\nonumber \\
&\lesssim (1+t)^{-\frac{n}{4\sigma}+\varepsilon_1(p^*)} \|(u, v)\|_{X(t)}^{p^*} \int_0^c \frac{\mu_1(s)}{s} ds\nonumber \\
&\lesssim (1+t)^{-\frac{n}{4\sigma}+\varepsilon_1(p^*)} \|(u, v)\|_{X(t)}^{p^*},\label{eq10thr1.1}
\end{align}
where we have employed Lemma \ref{integralmoduluslemma} with $\alpha_1=n/4\sigma$, $\alpha_2=n/2\sigma$, $\beta_1=1$ and $\beta_2=0$ and the assumption \eqref{con2thr1.1}. For $k \in \{q^*, \infty\}$, combining \eqref{eq2thr1.1} and \eqref{eq8thr1.1}, then applying Lemma \ref{integralmoduluslemma} with
\[
\alpha_1 = \frac{n}{\sigma}\tron{1-\frac{1}{k}}-1,\quad  \alpha_2 = \frac{n}{2\sigma},\quad \beta_1 = 1,\quad \beta_2 = 0,
\]
we obtain
\begin{align}
\big\|u^{\mathrm{nl}}(t, \cdot)\big\|_{L^k} 
&\lesssim (1+t)^{\varepsilon_1(p^*)} \|(u, v)\|_{X(t)}^{p^*} \int_0^t (1+t-\tau)^{1-\frac{n}{\sigma}\left(1-\frac{1}{k}\right)} (1+\tau)^{-1} \mu_1\left(c(1+\tau)^{-\frac{n}{2\sigma}}\right) d\tau \nonumber\\
&\lesssim (1+t)^{1-\frac{n}{\sigma}\left(1-\frac{1}{k}\right)+\varepsilon_1(p^*)} \|(u, v)\|_{X(t)}^{p^*} \int_0^{\infty} (1+\tau)^{-1} \mu_1\left(c(1+\tau)^{-\frac{n}{2\sigma}}\right) d\tau \nonumber\\
&\lesssim (1+t)^{1-\frac{n}{\sigma}\left(1-\frac{1}{k}\right)+\varepsilon_1(p^*)} \|(u, v)\|_{X(t)}^{p^*}.\label{eq11thr1.1}
\end{align}
By the condition \eqref{critcurvecase1}, we deduce that
\[
p^*=\frac{2\sigma}{n}\tron{2+\frac{1}{q}}+\frac{2}{q}-1>\frac{4\sigma}{n}-1,
\]
which implies
\begin{equation}\label{eq_power_negative_thr1.1}
-\frac{n}{\sigma}+1+\varepsilon_1(p^*)<0.
\end{equation}
Combining \eqref{eq5thr1.1}, \eqref{eq9thr1.1} and \eqref{eq_power_negative_thr1.1}, one may estimate
\begin{align}
\big\||D|^{\sigma}v^{\mathrm{nl}}(t, \cdot)\big\|_{L^2} \lesssim & \|(u, v)\|_{X(t)}^{p^*}\int_0^t(1+t-\tau)^{-\frac{n}{4\sigma}-\frac{1}{2}}(1+\tau)^{-1} \mu_2\left(c(1+\tau)^{-\frac{n}{\sigma}+1+\varepsilon_1(p^*)}\right) d \tau\nonumber \\
\lesssim &(1+t)^{-\frac{n}{4\sigma}-\frac{1}{2}}\|(u, v)\|_{X(t)}^{p^*},\label{eq12thr1.1}
\end{align}
where we used Lemma \ref{integralmoduluslemma} with $\beta_1=0, \beta_2=1, \alpha=n/\sigma-1-\varepsilon_1(p^*)$. Analogously, one can prove the following estimate for $k\in\{2,\infty\}$:
\begin{align}
\big\|v^{\mathrm{nl}}(t, \cdot)\big\|_{L^k} \lesssim & \|(u, v)\|_{X(t)}^{p^*}\int_0^t(1+t-\tau)^{-\frac{n}{2\sigma}\tron{1-\frac{1}{k}}}(1+\tau)^{-1} \mu_2\left(c(1+\tau)^{-\frac{n}{\sigma}+1+\varepsilon_1(p^*)}\right) d \tau\nonumber \\
\lesssim &(1+t)^{-\frac{n}{2\sigma}\tron{1-\frac{1}{k}}}\|(u, v)\|_{X(t)}^{p^*}.\label{eq13thr1.1}
\end{align}
\noindent If \eqref{con2thr1.1} does not hold, then we assume \eqref{con3athr1.1}. So, we take
$$
\ell_1(\tau)=\left(\frac{\mu_1\left(c(1+\tau)^{-\gamma}\right)}{\mu_2\left(c(1+\tau)^{-\gamma}\right)}\right)^{\frac{1}{q^*+1}},\quad \ell_2(t)\equiv 1,\quad \varepsilon_1(p^*)=1-\frac{n}{2\sigma}(p^*-1),\quad \varepsilon_2(q^*)\equiv 0.
$$
By choosing $0<\gamma<n/\sigma-1-\varepsilon_1(p^*)$, we deduce from \eqref{eq3thr1.1} and \eqref{eq8thr1.1}, together with Lemma \ref{ell1ell2}, that
\begin{align}
& \big\||D|^{\sigma}u^{\mathrm{nl}}(t, \cdot)\big\|_{L^2}\nonumber\\
&\quad \lesssim\|(u, v)\|_{X(t)}^{p^*} \int_0^t(1+t-\tau)^{-\frac{n}{4\sigma}}(1+\tau)^{-1+\varepsilon_1(p^*)} \ell_1(\tau)\big(\ell_1(\tau)\big)^{-1} \mu_1\left(c(1+\tau)^{-\gamma}\right) d \tau\nonumber\\
&\quad \lesssim(1+t)^{\varepsilon_1(p^*)} \ell_1(t)\|(u, v)\|_{X(t)}^{p^*} \nonumber\\
&\qquad \times \int_0^t(1+t-\tau)^{-\frac{n}{4\sigma}}(1+\tau)^{-1}\big(\mu_1\left(c(1+\tau)^{-\gamma}\right)\big)^{\frac{q^*}{q^*+1}}\big(\mu_2\left(c(1+\tau)^{-\gamma}\right)\big)^{\frac{1}{q^*+1}} d \tau\nonumber\\
&\quad\lesssim(1+t)^{-\frac{n}{4\sigma}+\varepsilon_1(p^*)}\ell_1(t)\|(u, v)\|_{X(t)}^{p^*}\nonumber\\
&\qquad \times\int_0^t(1+\tau)^{-1}\left(\mu_1\left(c(1+\tau)^{-\gamma}\right)\right)^{\frac{q^*}{q^*+1}}\left(\mu_2\left(c(1+\tau)^{-\gamma}\right)\right)^{\frac{1}{q^*+1}} d \tau\nonumber\\
&\quad\lesssim (1+t)^{-\frac{n}{4\sigma}+\varepsilon_1(p^*)} \ell_1(t)\|(u, v)\|_{X(t)}^{p^*} \int_0^c \frac{1}{s}\big(\mu_1(s)\big)^{\frac{q^*}{q^*+1}}\big(\mu_2(s)\big)^{\frac{1}{q^*+1}} ds\nonumber\\
&\quad\lesssim(1+t)^{-\frac{n}{4\sigma}+\varepsilon_1(p^*)}\ell_1(t)\|(u, v)\|_{X(t)}^{p^*},\label{eq14thr1.1}
\end{align}
where we have employed Lemma \ref{integralmoduluslemma} with
\[
\alpha_1 = \frac{n}{4\sigma}, \quad \alpha_2 = \gamma, \quad \beta_1 = \frac{q^*}{q^*+1}, \quad \beta_2 = \frac{1}{q^*+1}.
\]
Similarly, we have the following estimates for $k\in \{q^*,\infty\}$:
\begin{align}
&\big\|u^{\mathrm{nl}}(t, \cdot)\big\|_{L^k}\nonumber\\
&\quad \lesssim\|(u, v)\|_{X(t)}^{p^*} \int_0^t(1+t-\tau)^{1-\frac{n}{\sigma}\tron{1-\frac{1}{k}}}(1+\tau)^{-1+\varepsilon_1(p^*)} \ell_1(\tau)\big(\ell_1(\tau)\big)^{-1} \mu_1\left(c(1+\tau)^{-\gamma}\right) d \tau\nonumber\\
&\qquad \lesssim(1+t)^{\varepsilon_1(p^*)} \ell_1(t)\|(u, v)\|_{X(t)}^{p^*} \nonumber\\
&\qquad \times \int_0^t(1+t-\tau)^{1-\frac{n}{\sigma}\tron{1-\frac{1}{k}}}(1+\tau)^{-1}\left(\mu_1\left(c(1+\tau)^{-\gamma}\right)\right)^{\frac{q^*}{q^*+1}}\left(\mu_2\left(c(1+\tau)^{-\gamma}\right)\right)^{\frac{1}{q^*+1}} d \tau\nonumber\\
&\quad \lesssim(1+t)^{1-\frac{n}{\sigma}\tron{1-\frac{1}{k}}+\varepsilon_1(p^*)}\ell_1(t)\|(u, v)\|_{X(t)}^{p^*}.\label{eq15thr1.1}
\end{align}
To estimate the term containing $v^{\mathrm{nl}}(t,\cdot)$, we can proceed as follows:
\begin{align}
&\big\||D|^{\sigma}v^{\mathrm{nl}}(t, \cdot)\big\|_{L^2}\lesssim\|(u, v)\|_{X(t)}^{q^*} \int_0^t(1+t-\tau)^{-\frac{n}{4\sigma}-\frac{1}{2}}(1+\tau)^{-1} \big(\ell_1(\tau)\big)^{q^*} \mu_2\left(c(1+\tau)^{-\gamma}\right) d\tau \nonumber\\
& \lesssim\|(u, v)\|_{X(t)}^{q^*}\int_0^t(1+t-\tau)^{-\frac{n}{4\sigma}-\frac{1}{2}}(1+\tau)^{-1}\big(\mu_1\left(c(1+\tau)^{-\gamma}\right)\big)^{\frac{q^*}{q^*+1}}\big(\mu_2\left(c(1+\tau)^{-\gamma}\right)\big)^{\frac{1}{q^*+1}} d \tau\nonumber\\
&\lesssim(1+t)^{-\frac{n}{4\sigma}-\frac{1}{2}}\|(u, v)\|_{X(t)}^{p^*}, \label{eq16thr1.1}
\end{align}
analogously, 
\begin{equation}\label{eq17thr1.1}
\big\|v^{\mathrm{nl}}(t,\cdot)\big\|_{L^k}\lesssim (1+t)^{-\frac{n}{2\sigma}\tron{1-\frac{1}{k}}}\|(u,v)\|^{p^*}_{X(t)} \quad \text{ for }k\in\{2,\infty\}.
\end{equation}
Summarizing, all estimates \eqref{eq10thr1.1}-\eqref{eq17thr1.1} are to verify \eqref{eq1thr1.1} in the case $p^*\leq p_{\rm c}$ and $q^*\geq q_{\rm c}$. \medskip

{\bf Next, let us consider the case $p^*>p_{\rm c}\text{ and }q^*<q_{\rm c}$}. Then, the critical curve \eqref{critcurve} becomes
\begin{equation}\label{critcurvecase2}
    \frac{p^*q^*+p^*+1}{2p^*q^*-p^*-1}=\frac{n}{2\sigma}.
\end{equation}
\noindent If the condition \eqref{con2thr1.1} holds, we choose $\ell_1(t)\equiv \ell_2(t)\equiv 1$, $\varepsilon_1(p^*)\equiv 0$, $\varepsilon_2(q^*)=1+q^*-(n/\sigma)(q^*-1)$. Then, one recognizes
\begin{equation}
-\frac{n}{2\sigma}\left(p^*-1\right)+p^*\varepsilon_2(q^*)=-1, \quad q^*-\frac{n}{\sigma}\left(q^*-1\right)=-1+\varepsilon_2(q^*).
\end{equation}
Therefore, we deduce from \eqref{eq6thr1.1} and \eqref{eq7thr1.1} that
\begin{align}
\big\||v(\tau, \cdot)|^{p^*} \mu_1\big(|v(\tau, \cdot)|\big)\big\|_{L^1\cap L^2}&\lesssim (1+\tau)^{-1}\mu_1\tron{c(1+\tau)^{-\frac{n}{2\sigma}}}\|(u,v)\|^{p^*}_{X(t)},\label{eq18thr1.1}\\
\big\||u(\tau, \cdot)|^{q^*} \mu_2\big(|u(\tau, \cdot)|\big)\big\|_{L^1 \cap L^2}&\lesssim (1+\tau)^{-1+\varepsilon_2(q^*)}\mu_2\tron{c(1+\tau)^{-\frac{n}{\sigma}+1}}\|(u,v)\|^{q^*}_{X(t)}.\label{eq19thr1.1}
\end{align}
After repeating some steps as we did in the previous case, we derive the inequality \eqref{eq1thr1.1}. If \eqref{con2thr1.1} does not hold, then we assume \eqref{con3athr1.1}. In this way, we set
$$
\ell_1(\tau)\equiv 1,\quad  \ell_2(\tau)=\left(\frac{\mu_2\left(c(1+\tau)^{-\gamma}\right)}{\mu_1\left(c(1+\tau)^{-\gamma}\right)}\right)^{\frac{1}{p^*+1}},\quad \varepsilon_1(p^*)\equiv 0,\quad  \varepsilon_2(q^*)=1+q^*-\frac{n}{\sigma}(q^*-1),
$$
where $\gamma$ is chosen to satisfy $0 < \gamma < n/\sigma-1$. Following the same strategy as in the previous case, we may conclude that
\begin{align}
& \big\||D|^{\sigma}u^{\mathrm{nl}}(t, \cdot)\big\|_{L^2}\lesssim\|(u, v)\|_{X(t)}^{p^*} \int_0^t(1+t-\tau)^{-\frac{n}{4\sigma}}(1+\tau)^{-1} \big(\ell_2(\tau)\big)^{p^*}\mu_1\left(c(1+\tau)^{-\gamma}\right) d \tau\nonumber\\
& \lesssim\|(u, v)\|_{X(t)}^{p^*}\int_0^t(1+t-\tau)^{-\frac{n}{4\sigma}}(1+\tau)^{-1}\left(\mu_1\left(c(1+\tau)^{-\gamma}\right)\right)^{\frac{1}{p^*+1}}\left(\mu_2\left(c(1+\tau)^{-\gamma}\right)\right)^{\frac{p^*}{p^*+1}} d \tau\nonumber\\
&\lesssim(1+t)^{-\frac{n}{4\sigma}}\|(u, v)\|_{X(t)}^{p^*}\int_0^t(1+\tau)^{-1}\left(\mu_1\left(c(1+\tau)^{-\gamma}\right)\right)^{\frac{1}{p^*+1}}\left(\mu_2\left(c(1+\tau)^{-\gamma}\right)\right)^{\frac{p^*}{p^*+1}} d \tau\nonumber\\
&\lesssim (1+t)^{-\frac{n}{4\sigma}}\|(u, v)\|_{X(t)}^{p^*} \int_0^c \frac{1}{s}\big(\mu_1(s)\big)^{\frac{1}{p^*+1}}\big(\mu_2(s)\big)^{\frac{p^*}{p^*+1}} d s\nonumber\\
&\lesssim(1+t)^{-\frac{n}{4\sigma}}\|(u, v)\|_{X(t)}^{p^*},\label{eq20thr1.1}
\end{align}
where we have applied Lemma \ref{integralmoduluslemma} after choosing
\[
\alpha_1=\frac{n}{4\sigma},\quad \alpha_2=\gamma,\quad \beta_1=\frac{1}{p^*+1},\quad  \beta_2=\frac{p^*}{p^*+1}.
\]
For $k\in \{q^*,\infty\}$, since $(n/\sigma)\tron{1-1/k}-1\leq 1$, a similar argument leads to
\begin{align}
\big\|u^{\mathrm{nl}}(t, \cdot)\big\|_{L^k}&\lesssim\|(u, v)\|_{X(t)}^{p^*} \int_0^t(1+t-\tau)^{1-\frac{n}{\sigma}\tron{1-\frac{1}{k}}}(1+\tau)^{-1}\big(\ell_2(\tau)\big)^{p^*}\mu_1\left(c(1+\tau)^{-\gamma}\right) d \tau\nonumber\\
&\lesssim(1+t)^{1-\frac{n}{\sigma}\tron{1-\frac{1}{k}}}\|(u, v)\|_{X(t)}^{p^*}.\label{eq21thr1.1}
\end{align}
Notice that $(1+t)^{\varepsilon_2(q^*)}\ell_2(t)$ is an increasing function (see Lemma \ref{ell1ell2}), one has 
\begin{align}
&\big\||D|^{\sigma}v^{\mathrm{nl}}(t, \cdot)\big\|_{L^2}\nonumber\\
&\lesssim\|(u, v)\|_{X(t)}^{q^*} \int_0^t(1+t-\tau)^{-\frac{n}{4\sigma}-\frac{1}{2}}(1+\tau)^{-1+\varepsilon_2(q^*)}\ell_2(\tau)\big(\ell_2(\tau)\big)^{-1}\mu_2\left(c(1+\tau)^{-\gamma}\right) d\tau \nonumber\\
& \lesssim(1+\tau)^{\varepsilon_2(q^*)}\ell_2(t)\|(u, v)\|_{X(t)}^{q^*}\nonumber\\
&\quad \quad \times \int_0^t(1+t-\tau)^{-\frac{n}{4\sigma}-\frac{1}{2}}(1+\tau)^{-1}\left(\mu_1\left(c(1+\tau)^{-\gamma}\right)\right)^{\frac{1}{p^*+1}}\left(\mu_2\left(c(1+\tau)^{-\gamma}\right)\right)^{\frac{p^*}{p^*+1}} d \tau\nonumber\\
&\lesssim(1+t)^{-\frac{n}{4\sigma}-\frac{1}{2}+\varepsilon_2(q^*)}\|(u, v)\|_{X(t)}^{q^*}\int_0^t(1+\tau)^{-1}\left(\mu_1\left(c(1+\tau)^{-\gamma}\right)\right)^{\frac{1}{p^*+1}}\left(\mu_2\left(c(1+\tau)^{-\gamma}\right)\right)^{\frac{p^*}{p^*+1}} d \tau\nonumber\\
&\lesssim(1+t)^{-\frac{n}{4\sigma}-\frac{1}{2}+\varepsilon_2(q^*)}\|(u, v)\|_{X(t)}^{q^*},\label{eq22thr1.1}
\end{align}
where we have employed Lemma \ref{integralmoduluslemma} with
$$
\alpha_1=\frac{n}{4\sigma}+\frac{1}{2},\quad \alpha_2=\gamma,\quad \beta_1=\frac{q^*}{q^*+1},\quad \beta_2=\frac{1}{q^*+1}.
$$
By the same way, one achieves 
\begin{equation}\label{eq23thr1.1}
\big\|v^{\mathrm{nl}}(t,\cdot)\big\|_{L^k}\lesssim (1+t)^{-\frac{n}{2\sigma}\tron{1-\frac{1}{k}}}\|(u,v)\|^{p^*}_{X(t)}\quad\text{ for }k\in\{2,\infty\}.
\end{equation}
The link of all estimates \eqref{eq20thr1.1}-\eqref{eq23thr1.1} leads to the inequality \eqref{eq1thr1.1}. \medskip

Next, let us prove the inequality \eqref{eqbanachfixedpoint2}. For two elements $(u, v)$ and $(\tilde{u}, \tilde{v})$ from $X(t)$, it is obvious to see that
$$
\Psi(u, v)(t,x)-\Psi(\tilde{u}, \tilde{v})(t,x)=\left(u^{\mathrm{nl}}(t,x)-\tilde{u}^{\mathrm{nl}}(t,x), v^{\mathrm{nl}}(t,x)-\tilde{v}^{\mathrm{nl}}(t,x)\right).
$$
Then, we employ Proposition \ref{Linear estimates hyp} to derive the following estimates for $k \in \{2, \infty\}$:
\begin{align*}
& \big\|\big(u^{\mathrm{nl}}-\tilde{u}^{\mathrm{nl}}\big)(t, \cdot)\big\|_{L^k} \\
& \lesssim \int_0^t(1+t-\tau)^{-\frac{n}{\sigma}\tron{1-\frac{1}{k}}+1}\big\||v(\tau, \cdot)|^{p^*} \mu_1\big(|v(\tau, \cdot)|\big)-|\tilde{v}(\tau, \cdot)|^{p^*} \mu_1(|\tilde{v}(\tau, \cdot)|)\big\|_{L^1 \cap L^2} d \tau
\end{align*}
and 
\begin{align*}
& \big\|\big(u^{\mathrm{nl}}-\tilde{u}^{\mathrm{nl}}\big)(t, \cdot)\big\|_{\dot{H}^{\sigma}} \\
& \lesssim \int_0^t(1+t-\tau)^{-\frac{n}{4\sigma}}\big\||v(\tau, \cdot)|^{p^*} \mu_1\big(|v(\tau, \cdot)|\big)-|\tilde{v}(\tau, \cdot)|^{p^*} \mu_1(|\tilde{v}(\tau, \cdot)|)\big\|_{L^1 \cap L^2} d \tau.
\end{align*}
For $k \in \{q^*, \infty\}$, the estimates from Proposition \ref{Linear estimate par} imply
\begin{align*}
&\big\|\big(v^{\mathrm{nl}}-\tilde{v}^{\mathrm{nl}}\big)(t, \cdot)\big\|_{L^k} \\
&\quad \lesssim \int_0^t (1+t-\tau)^{-\frac{n}{2\sigma}\left(1-\frac{1}{k}\right)} \big\| |u(\tau,\cdot)|^{q^*} \mu_2(|u(\tau,\cdot)|) - |\tilde{u}(\tau,\cdot)|^{q^*} \mu_2(|\tilde{u}(\tau,\cdot)|) \big\|_{L^1 \cap L^2} d\tau
\end{align*}
and
\begin{align*}
& \big\|\big(v^{\mathrm{nl}}-\tilde{v}^{\mathrm{nl}}\big)(t, \cdot)\big\|_{\dot{H}^{\sigma}} \\
&\quad \lesssim \int_0^t(1+t-\tau)^{-\frac{n}{4\sigma}-\frac{1}{2}}\big\||u(\tau, \cdot)|^{q^*} \mu_2\big(|u(\tau, \cdot)|\big)-|\tilde{u}(\tau, \cdot)|^{q^*} \mu_2(|\tilde{u}(\tau, \cdot)|)\big\|_{L^1 \cap L^2} d \tau.
\end{align*}
The application of the mean value theorem gives the following integral representation:
\begin{align*}
&|v(\tau,x)|^{p^*} \mu_1(|v(\tau,x)|)-|\tilde{v}(\tau,x)|^{p^*} \mu_1(|\tilde{v}(\tau,x)|)\\
&\qquad =(v(\tau,x)-\tilde{v}(\tau,x)) \int_{0}^{1}d_{|v|} \Phi_p(\omega v(\tau,x)+(1-\omega) \tilde{v}(\tau,x)) d \omega,
\end{align*}
where $\Phi_p(s)$ is defined as in \eqref{critnonlinearity}. Since the condition \eqref{con1thr1.1} holds, one gets
$$
d_{|v|} \Phi_p(v)=p^*|v(t,x)|^{p^*-1} \mu_1(|v(t,x)|)+|v(t,x)|^{p^*} d_{|v|} \mu_1(|v(t,x)|) \lesssim|v(t,x)|^{p^*-1} \mu_1(|v(t,x)|).
$$
Thus, it follows that
$$
\begin{aligned}
& \big||v(\tau,x)|^{p^*} \mu_1(|v(\tau,x)|)-|\tilde{v}(\tau,x)|^{p^*} \mu_1(|\tilde{v}(\tau,x)|) \big|\\
& \quad \lesssim |v(\tau,x)-\tilde{v}(\tau,x)| \int_0^1 |\omega v(\tau,x)+(1-\omega) \tilde{v}(\tau,x)|^{p^*-1} \mu_1\big(|\omega v(\tau,x)+(1-\omega) \tilde{v}(\tau,x)|\big) d \omega .
\end{aligned}
$$
Similarly, we also obtain
$$
\begin{aligned}
& \big||u(\tau,x)|^{q^*} \mu_2(|u(\tau,x)|)-|\tilde{u}(\tau,x)|^{q^*} \mu_2(|\tilde{u}(\tau,x)|) \big| \\
& \quad \lesssim|u(\tau,x)-\tilde{u}(\tau,x)| \int^{1}_0 |\omega u(\tau,x)+(1-\omega) \tilde{u}(\tau,x)|^{q^*-1} \mu_2\big(|\omega u(\tau,x)+(1-\omega) \tilde{u}(\tau,x)|\big) d \omega .
\end{aligned}
$$
By the aid of Hölder's inequality and applying the same arguments as in the proof of inequality \eqref{eq1thr1.1}, we arrive at the inequality \eqref{eqbanachfixedpoint2}. Hence, the proof of Theorem \ref{thr1.1} is completed.

\section{Blow-up result}
\subsection{Construction of test functions}\label{subsec4.1}
We now introduce the test functions that will be used in the proof of Theorem \ref{thr1.2}. We define the functions $\eta \in \mathcal{C}_0^{\infty}([0,\infty))$ and $\eta^T$ fulfilling 
\[
\eta(t)=
\left\{
\begin{aligned}
    &1 &&\text{ if }0\leq t\leq 1/2,\\
    &\text{decreasing}&&\text{ if }1/2\leq t\leq 1,\\
    &0 &&\text{ if }t\geq 1,
\end{aligned}
\right.
\text{ and }
\eta^T(t)=
\left\{
\begin{aligned}
&0 &&\text{ if }0\leq t<1/2,\\
&\eta(t) &&\text{ if }t\geq 1/2.
\end{aligned}
\right.
\]
Moreover, the functions $\varphi\in H^{2\sigma}$ and $\varphi^X$ are given by 
\[
\varphi(x):=
\left\{
\begin{aligned}
&1&&\text{ for $|x|\leq 1$},\\
&\sqrt[4]{1+(|x|-1)^{4}}&&\text{ for $|x|\geq 1$},
\end{aligned}
\right.
\text{ and }
\varphi^X(x):=
\left\{
\begin{aligned}
&0&&\text{ for }|x|\leq 1,\\
&\varphi(x)&&\text{ for }|x|\geq 1.
\end{aligned}
\right.
\]
For a large parameter $R\in [0,\infty)$, we introduce the functions
\[
\eta_R(t):=\eta\tron{R^{-1}t}\text{ and }\eta^T_{R}(t):=\eta^T\tron{R^{-1}t},
\]
together with the following space-dependent functions possessing different scaling orders:
\[
\varphi_{j,R}(x):=\varphi\tron{R^{-\frac{1}{j\sigma}}x}\text{ and }\varphi^X_{j,R}(x):=\varphi^X\tron{R^{-\frac{1}{j\sigma}}x}
\]
for $j=1,2$. Let $\kappa=\sigma-\flo{\sigma}$ if $\sigma$ is fractional and $\kappa=\sigma$ if $\sigma$ is an integer. For $j\in \{1,2\}$, we define the test functions
\begin{equation}\label{testfunc}
\begin{aligned}
&\phi_{j,R}(t,x):=\left[\varphi_{j,R}(x)\right]^{n+2\kappa}[\eta_{R}(t)]^{\nu+2},\\
&\phi^{X}_{j,R}(t,x):=\left[\varphi^X_{j,R}(x)\right]^{n+2\kappa}[\eta_{R}(t)]^{\nu},\\
&\phi^{T}_{j,R}(t,x):=\left[\varphi_{j,R}(x)\right]^{n+2\kappa}[\eta^T_{R}(t)]^{\nu},
\end{aligned}
\end{equation}
where the parameter $\nu>0$ will be fixed later, having the following properties:
\begin{equation}\label{domain}
\begin{aligned}
\text{supp}\phi_R(t,x)\subset Q_R&:=\left\{(t,x):x\in \mathbb{R}^n,t\in[0,R]\right\},\\
\text{supp}\phi^{X}_{j,R}(t,x)\subset Q^1_{j,R}&:=\left\{(t,x):|x|\geq R^{\frac{1}{j\sigma}},t\in [0,R]\right\},\\
\text{supp}\phi^{T}_{j,R}(t,x)\subset Q^2_{R}&:=\left\{(t,x):x\in \mathbb{R}^n,t\in\left[R/2,R\right]\right\},\\
Q^3_{j,R}&:=\left\{(t,x):|x|\leq R^{\frac{1}{j\sigma}},t\leq R/2\right\}.
\end{aligned}
\end{equation}
A straightforward calculation gives
\begin{align}
|\partial_t \phi_{j,R}(t,x)| &=[\varphi_{j,R}(x)]^{n+2\kappa}\left|\partial_t[\eta_R(t)^{\nu+2}]\right| \nonumber \\
&\lesssim R^{-1}[\varphi_{j,R}(x)]^{n+2\kappa}[\eta^{T}_R(t)]^{\nu+1} \lesssim R^{-1}\phi^{T}_{j,R}(t,x), \label{testfunceq1} \\
\big|\partial^2_t\phi_{j,R}(t,x)\big|&=[\varphi_{j,R}(x)]^{n+2\kappa}\left|\partial^2_t[\eta_R(t)]^{\nu+2}\right|\nonumber\\
&=[\varphi_{j,R}(x)]^{n+2\kappa}\left|(\nu+2)(\nu+1)[\eta^{T}_R(t)]^{\nu}[\eta^{\prime}_R(t)]^2+(\nu+2)[\eta^{T}_R(t)]^{\nu+1}[\eta^{\prime\prime}_R(t)]\right|\nonumber\\
&\lesssim R^{-2}[\varphi_{j,R}(x)]^{n+2\kappa}[\eta^{T}_R(t)]^{\nu}\lesssim R^{-2}\phi^{T}_{j,R}(t,x) \label{testfunceq2}
\end{align}
for all $(t,x)\in Q^2_R$. Furthermore, if $\sigma$ is an integer, then $\kappa=\sigma$. We use Lemma \ref{lem_Laplace_integer} to conclude that
\begin{align}
&\big|(-\Delta)^{\sigma} \phi_{j,R}(t,x)\big|=\left|\eta_R(t)R^{-1}((-\Delta)^{\sigma}[\varphi]^{n+2\kappa})\tron{R^{-\frac{1}{j\sigma}}x}\right|\nonumber\\
&\qquad\lesssim R^{-2/j}[\eta_R(t)]^{\nu+2}[\varphi^{X}_{j,R}(x)]^{n+2\kappa+2\sigma}\lesssim R^{-2/j}\phi^{X}_{j,R}(t,x),\;\text{ for a.e. } (t,x)\in Q_R.\label{testfunceq3_integer}
\end{align}
If $\sigma$ is fractional, then $\kappa=\sigma-\flo{\sigma}$. Lemma \ref{lem_Laplace_fractional} implies the estimate
\begin{align}\label{testfunceq3_fractional}
\big|(-\Delta)^{\sigma} \phi_{j,R}(t,x)\big|&\lesssim R^{-2/j}[\eta_R(t)]^{\nu+2}[\varphi_{j,R}(x)]^{n+2\kappa} \lesssim R^{-2/j}\phi_{j,R}(t,x)
\end{align}
for a.e. $(t,x)\in Q_R$. Similarly, we obtain the following estimate for a.e. $(t,x)\in Q^2_R$:
\begin{align}
\big|(-\Delta)^{\sigma}\partial_t\phi_{1,R}(t,x)\big| &\lesssim \left|R^{-1}\partial_t\eta^{T}_R(t)[\eta^{T}_R(t)]^{\nu+1}\big((-\Delta)^{\sigma}[\varphi]^{n+2\sigma}\big)\big(R^{-1/\sigma}x\big)\right|\nonumber\\
&\lesssim R^{-3}[\eta^{T}_R(t)]^{\nu+1}[\varphi_{1,R}(x)]^{n+2\kappa}\lesssim R^{-3}\phi^{T}_{1,R}(t,x).\label{testfunceq4}
\end{align}
In addition, the following version of Jensen's inequality comes into play.
\begin{lemma}[Generalized Jensen's inequality, see \cite{EbeGirRei2020}] \label{Jenseninequality}
Let $\eta=\eta(x)$ be a defined and nonnegative function almost everywhere on $\Omega$, provided that $\eta$ is positive in a set of positive measure. Then, for each convex function $h$ on $\mathbb{R}$ the following inequality holds:
$$
h\left(\frac{\int_{\Omega} f(x) \eta(x) d x}{\int_{\Omega} \eta(x) d x}\right) \leq \frac{\int_{\Omega} h\big(f(x)\big)\eta(x) d x}{\int_{\Omega} \eta(x) d x},
$$
where $f$ is any nonnegative function satisfying all the above integrals are meaningful.
\end{lemma}

Now, we are ready to present the proof of Theorem \ref{thr1.2} in the next sequel.
\subsection{Proof of Theorem 1.2}\label{subsec4.2}

{\bf First, let us consider the case $p\leq p_{\rm c}$ and $q\geq q_{\rm c}$}. Our proof is organized as follows:
\begin{itemize}
\item Step 1: Define the functionals $I_R, J_R, I^*_R, J^*_R$ and derive preliminary estimates on $Q^1_{j,R}, Q^2_{R}$, and $Q^3_{j,R}$ with $j=1,2$.
\item Step 2: Introduce auxiliary functions $g_p, g_q, G_p, G_q$ and establish a system of ordinary differential inequalities for these functions.
\item Step 3: Decompose the resulting system of ordinary differential inequalities into the terms $A_i$ and $B_i$ for $i=1,2,3$, and derive the corresponding estimates for each term.
\item Step 4: Consolidate the obtained estimates in Step 3 to complete our proof.
\end{itemize}
The full details of the proof are presented as follows:

\noindent {\bf $\bullet$ Step 1:} First of all, we define the following two functionals:
\[
\begin{aligned}
I_{R}&:=\int_{0}^{\infty}\int_{\mathbb{R}^n}|v(t,x)|^{p^*}\mu_1(|v(t,x)|)\phi_{1,R}(t,x)\;dxdt \\
&=\int_{Q_{R}}|v(t,x)|^{p^*}\mu_1(|v(t,x)|)\phi_{1,R}(t,x)d(t,x),\\    J_{R}&:=\int_{0}^{\infty}\int_{\mathbb{R}^n}|u(t,x)|^{q^*}\mu_2(|u(t,x)|)\phi_{2,R}(t,x)\;dxdt \\
&=\int_{Q_{R}}|u(t,x)|^{q^*}\mu_2(|u(t,x)|)\phi_{2,R}(t,x)d(t,x),  
\end{aligned}
\]
where the functions $\phi_{1,R}(t,x)$, $\phi_{2,R}(t,x)$ are defined as in \eqref{testfunc}. Let us assume that $(u,v)=(u(t,x),v(t,x))$ is a global (in time) Sobolev solution to \eqref{mainsystem} in sense of Definition \ref{def_weaksolution}. We multiply these equations on the left-hand sides of \eqref{mainsystem} by $\phi_{1,R}(t,x)$, $\phi_{2,R}(t,x)$ to arrive at
\begin{align}
0\leq I_{R}&=-\int_{\mathbb{R}^n}u_1(x)\phi_{1,R}(0,x)\;dx\nonumber\\
&\quad +\int_{Q_{R}}u(t,x)\big(\partial^2_t \phi_{1,R}(t,x)+(-\Delta)^{\sigma}\phi_{1,R}(t,x)-(-\Delta)^{\sigma}\partial_t \phi_{1,R}(t,x)\big)\;d(t,x)\nonumber\\
&=:-\mathcal{D}(u_1)+I^*_{R},\label{thr1.2esI1R}\\
0\leq J_{R}&=-\int_{\mathbb{R}^n}v_1(x)\phi_{2,R}(0,x)\;dx \nonumber\\
&\quad+\int_{Q_{R}}v(t,x)\big(\partial^2_t \phi_{2,R}(t,x)+(-\Delta)^{\sigma}\phi_{2,R}(t,x)-\partial_t  \phi_{2,R}(t,x)\big)\;d(t,x)\nonumber\\
&=:-\mathcal{D}(v_1)+J^*_{R}.\label{thr1.2esJ2R}
\end{align}
From the estimates \eqref{testfunceq2}, \eqref{testfunceq3_integer}, \eqref{testfunceq3_fractional} and \eqref{testfunceq4}, we obtain
\begin{equation}\label{thr1.2esI1R*}
\begin{aligned}
\left|I_{R}^*\right|\lesssim R^{-2}\Bigg( \int_{Q^1_{1,R}}|u(t,x)|\left(\phi^{X}_{1,R}(t,x)\right)d(t,x)&+\int_{Q^2_{R}}|u(t,x)|\tron{\phi^{T}_{1,R}(t,x)}d(t,x)\\
&\qquad+\int_{Q^3_{1,R}}|u(t,x)|d(t,x)\Bigg),
\end{aligned}   
\end{equation}
\begin{equation}\label{thr1.2esJ2R*}
\begin{aligned}
\left|J_{R}^*\right|\lesssim  R^{-1}\Bigg(\int_{Q^1_{2,R}}|v(t,x)|\left(\phi^{X}_{2,R}(t,x)\right)d(t,x)&+\int_{Q^2_{R}}|v(t,x)|\left(\phi^{T}_{2,R}(t,x)\right)d(t,x)\\
&\qquad+\int_{Q^3_{2,R}}|v(t,x)|d(t,x)\Bigg).
\end{aligned}
\end{equation}
We recall the functions $\Phi_p(s)$ and $\Phi_q(s)$ defined as in \eqref{critnonlinearity}. It is clear to verify that $\Phi_p$ and $\Phi_q$ are convex on $(0, c_0]$ for a sufficiently small $c_0 > 0$ even on $[0, \infty)$. For any $0<\delta<1$, the application of Lemma \ref{Jenseninequality} with $h(s)=\Phi_q(s), f(t,x)=|u(t,x)|\big(\phi^{X}_{1,R}(t,x)\big),\eta(t,x)=(\phi^{X}_{1,R}(t,x))^{\delta}$ and $\Omega \equiv  Q^1_{1,R}$ leads to the following estimate:
\[
\begin{aligned}
&\Phi_q\left(\frac{\int_{Q^1_{1,R}}|u(t,x)|\left(\phi^{X}_{1,R}(t,x)\right)  d(t,x)}{\int_{Q^1_{1,R}}[\phi^{X}_{1,R}(t,x)]^{\delta}d(t,x)}\right) \\
&\quad \quad \leq \frac{\int_{Q^1_{1,R}}\Phi_q\left(|u(t,x)|\left(\phi^{X}_{1,R}(t,x)\right)^{1-\delta}\right)(\phi^{X}_{1,R}(t,x))^{\delta}d(t,x)}{\int_{Q^1_{1,R}}[\phi^{X}_{1,R}(t,x)]^{\delta}d(t,x)}.
\end{aligned}
\]
By a change of variables $\tilde{t}=R^{-1}t, \tilde{x}=R^{-1/\sigma} x$, it holds
$$
\int_{Q^1_{1,R}}[\phi^{X}_{1,R}(t,x)]^{\delta}\;d(t,x)\sim R^{\frac{n}{\sigma}+1}\int_{1}^{\infty}[\varphi(\tilde{x})]^{\delta(n+2\kappa)}|\tilde{x}|^{n-1}\;d|\tilde{x}|\int_{0}^{1}[\eta(\tilde{t})]^{\delta \nu}\;d\tilde{t}\sim R^{\frac{n}{\sigma}+1}
$$
for any $\delta\in (n/(n+2\kappa),1)$. Combining this with the increasing property of $\mu_2$, we obtain
\begin{align*}
&\Phi_q\left(\frac{\int_{Q^1_{1,R}}|u(t,x)|\left(\phi^{X}_{1,R}(t,x)\right)d(t,x)}{CR^{\frac{n}{\sigma}+1}}\right)\nonumber \\
&\qquad\lesssim \frac{\int_{Q^1_{1,R}}\Phi_q(|u(t,x)|)\left(\phi^{X}_{1,R}(t,x)\right)^{q^*+(1-q^*)\delta}d(t,x)}{R^{\frac{n}{\sigma}+1}}.
\end{align*}
Since the function $\Phi_q=\Phi_q(s)$ is also a strictly increasing function on $[0, \infty)$, one has
\begin{equation}\label{thr1.2_eq1}
\begin{aligned}
&\int_{Q^1_{1,R}}|u(t,x)|\left(\phi^{X}_{1,R}(t,x)\right)d(t,x)\\
&\qquad\lesssim R^{\frac{n}{\sigma}+1} \Phi_q^{-1}\left(\frac{\int_{Q^1_{1,R}}\Phi_q(|u(t,x)|)\left(\phi^{X}_{1,R}(t,x)\right)^{q^*+(1-q^*)\delta}d(t,x)}{CR^{\frac{n}{\sigma}+1}}\right).
\end{aligned}
\end{equation}
The same argument yields
\begin{equation}\label{thr1.2_eq2}
\begin{aligned}
&\int_{Q^1_{2,R}}|v(t,x)|\left(\phi^{X}_{2,R}(t,x)\right)d(t,x)\\
&\qquad \lesssim R^{\frac{n}{2\sigma}+1} \Phi_p^{-1}\left(\frac{\int_{Q^1_{2,R}}\Phi_p(|v(t,x)|)\left(\phi^{X}_{2,R}(t,x)\right)^{p^*+(1-p^*)\delta}d(t,x)}{CR^{\frac{n}{2\sigma}+1}}\right).
\end{aligned}
\end{equation}
In a similar manner, we obtain analogous estimates for $Q^2_{R}$ as follows:
\begin{equation}\label{thr1.2_eq3}
\begin{aligned}
&\int_{Q^2_{R}}|u(t,x)|\left(\phi^{T}_{1,R}(t,x)\right)d(t,x)\\
&\qquad\lesssim R^{\frac{n}{\sigma}+1} \Phi_q^{-1}\left(\frac{\int_{Q^2_{R}}\Phi_q(|u(t,x)|)\left(\phi^{T}_{1,R}(t,x)\right)^{q^*+(1-q^*)\delta}d(t,x)}{CR^{\frac{n}{\sigma}+1}}\right),
\end{aligned}
\end{equation}
\begin{equation}\label{thr1.2_eq4}
\begin{aligned}
&\int_{Q^2_{R}}|v(t,x)|\left(\phi^{T}_{2,R}(t,x)\right)d(t,x)\\
&\qquad\lesssim R^{\frac{n}{2\sigma}+1} \Phi_p^{-1}\left(\frac{\int_{Q^2_{R}}\Phi_p(|v(t,x)|)\left(\phi^{T}_{2,R}(t,x)\right)^{p^*+(1-p^*)\delta}d(t,x)}{CR^{\frac{n}{2\sigma}+1}}\right)
\end{aligned}
\end{equation}
by using the calculation
\[
\int_{Q^2_{R}}\left[\phi^{T}_{j,\lambda}(t,x)\right]^{\delta}d(t,x)=R^{\frac{n}{j\sigma}+1}\int_{0}^{\infty}[\varphi(\tilde{x})]^{\delta(n+2\kappa)}|\tilde{x}|^{n-1}\;d|\tilde{x}|\int_{1/2}^{1}[\eta(\tilde{t})]^{\delta \nu}\;d\tilde{t}\sim CR^{\frac{n}{j\sigma}+1}.
\]
for any $\delta\in (n/(n+2\kappa),1)$ and $j\in \{1,2\}$. To obtain the estimates on $Q^3_{j,R}$ with $j\in \{1,2\}$, we first define the following functions:
\[
\xi(x)=
\begin{cases}
    |x| &\text{ for }|x|\leq 1, \\
    0 &\text{ for }|x|>1, 
\end{cases}
\quad \text{ and }\quad 
\xi_{j,R}(x)=\xi\tron{R^{-\frac{1}{j\sigma}}x}.
\]
Notice that for any sufficiently small positive parameter $a$, we have
\begin{align}
\int_{Q^3_{j,R}}[\xi_{j,R}(x)]^{-\frac{a}{q^*-1}}d(t,x)&=R^{\frac{n}{j\sigma}+1}\int_{0}^{1/2}d\tilde{t}\int_{0}^{1}|\tilde{x}|^{n-\frac{a}{q^*-1}-1}\;d|\tilde{x}|\sim CR^{\frac{n}{j\sigma}+1}.\label{thr1.2_eq5}
\end{align}
Employing Lemma \ref{Jenseninequality} again for $h(s)=\Phi_q(s)$, $f(t,x)=|u(t,x)|$, $\eta(x)=\left[\xi_{1,R}(x)\right]^{-\frac{a}{q^*-1}}$ and $\Omega\equiv Q^3_{1,R}$, we get
\[
\Phi_q\left(\frac{\int_{Q^3_{1,R}}|u(t,x)|d(t,x)}{\int_{Q^3_{1,R}}[\xi_{1,R}(x)]^{-\frac{a}{q^*-1}}d(t,x)}\right) \leq \frac{\int_{Q^3_{1,R}}\Phi_q\left(|u(t,x)|[\xi_{1,R}(x)]^{\frac{a}{q^*-1}}\right)[\xi_{1,R}(x)]^{-\frac{a}{q^*-1}}d(t,x)}{\int_{Q^3_R}[\xi_{1,R}(x)]^{-\frac{a}{q^*-1}}d(t,x)}.
\]
The combination of this with \eqref{thr1.2_eq5}, it holds
\begin{equation}\label{thr1.2_eq7}
\int_{Q^3_{1,R}}|u(t,x)|d(t,x)\lesssim R^{\frac{n}{\sigma}+1}\Phi^{-1}_q\tron{\frac{\int_{Q^3_{1,R}}\Phi_q\tron{|u(t,x)|}[\xi_{1,R}(x)]^{a}}{CR^{\frac{n}{\sigma}}}}.
\end{equation}
In a similar way, one also arrives at
\begin{equation}\label{thr1.2_eq8}
\int_{Q^3_{2,R}}|v(t,x)|d(t,x)\lesssim R^{\frac{n}{2\sigma}+1}\Phi^{-1}_p\tron{\frac{\int_{Q^3_{2,R}}\Phi_p\tron{|v(t,x)|}[\xi_{2,R}(x)]^{a}}{CR^{\frac{n}{2\sigma}}}}.
\end{equation}
\noindent $\bullet$ {\bf Step 2:} We define the following auxiliary functions for $\lambda\in (0,R)$:
\begin{align*}
    &g_{1,q}(\lambda):=\int_{Q^1_{1,R}} \Phi_q(|u(t,x)|)\left(\phi^{X}_{1,\lambda}(t,x)\right)^{q^*+(1-q^*)\delta}d(t,x)\text{ and }G_{1,q}(s):=\int_0^s g_{1,q}(\lambda)\lambda^{-1} d \lambda,\\
    &g_{1,p}(\lambda):=\int_{Q^1_{2,R}} \Phi_p(|u(t,x)|)\left(\phi^{X}_{2,\lambda}(t,x)\right)^{p^*+(1-p^*)\delta}d(t,x)\text{ and }G_{1,p}(s):=\int_0^s g_{1,p}(\lambda)\lambda^{-1} d \lambda,\\
    &g_{2,q}(\lambda):=\int_{Q^2_{R}} \Phi_q(|u(t,x)|)\left(\phi^{T}_{1,\lambda}(t,x)\right)^{q^*+(1-q^*)\delta}d(t,x)\text{ and }G_{2,q}(s):=\int_0^s g_{2,q}(\lambda)\lambda^{-1} d \lambda,\\
    &g_{2,p}(\lambda):=\int_{Q^2_{R}} \Phi_p(|v(t,x)|)\left(\phi^{T}_{2,\lambda}(t,x)\right)^{p^*+(1-p^*)\delta}d(t,x)\text{ and }G_{2,p}(s):=\int_0^s g_{2,p}(\lambda)\lambda^{-1} d \lambda,\\
    &g_{3,q}(\lambda):=\int_{Q^3_{1,R}} \Phi_q(|u(t,x)|)[\xi_{1,\lambda}(x)]^{a}d(t,x)\text{ and }G_{3,q}(s):=\int_0^s g_{3,q}(\lambda)\lambda^{-1} d \lambda,\\
    &g_{3,p}(\lambda):=\int_{Q^3_{2,R}} \Phi_p(|v(t,x)|)[\xi_{2,\lambda}(x)]^{a}d(t,x)\text{ and }G_{3,p}(s):=\int_0^s g_{3,q}(\lambda)\lambda^{-1} d \lambda,
\end{align*}
and
\[
g_\gamma(\lambda):=\sum_{i=1}^{3}g_{i,\gamma}(\lambda),\quad G_\gamma(\lambda):=\sum_{i=1}^{3}G_{i,\gamma}(\lambda)
\]
with $\gamma\in \{p,q\}$. Moreover, the following relations hold:
$$
g_{i,\gamma}(s)=G^{\prime}_{i,\gamma}(s)s \quad\text{ and }\quad g_\gamma(s)=G^{\prime}_\gamma(s)s
$$
for $i=1,2,3$ and $\gamma\in \{p,q\}$. Since $\Phi_q^{-1}$ and $\Phi_p^{-1}$ are concave functions, combining the estimates from \eqref{thr1.2_eq1} to \eqref{thr1.2_eq4} and \eqref{thr1.2_eq7}, \eqref{thr1.2_eq8}, we arrive at
\begin{align}
    I_{R}+\mathcal{D}(u_1)&\lesssim R^{\frac{n}{\sigma}-1}\Phi_q^{-1}\tron{\frac{G^{\prime}_q(R)}{CR^{\frac{n}{\sigma}}}},\label{upperIR} \\
    J_{R}+\mathcal{D}(v_1)&\lesssim R^{\frac{n}{2\sigma}}\Phi_p^{-1}\tron{\frac{G^{\prime}_p(R)}{CR^{\frac{n}{2\sigma}}}}\label{upperJR}. 
\end{align}
Next, we are going to estimate the lower bound of $I_R$ and $J_R$. We can express $G_q(R)$ as follows:
\begin{equation}\label{thr1.2_eq9}
\begin{aligned}
G_q(R)&=\int_{Q^1_{1,R}} \Phi_q(|u(t,x)|)\tron{\int_{0}^{R}\tron{\left(\phi^{X}_{1,\lambda}(t,x)\right)^{q^*+(1-q^*)\delta}\lambda^{-1}\;d \lambda}}d(t,x)\\
&\qquad+\int_{Q^2_{R}} \Phi_q(|u(t,x)|)\tron{\int_{0}^{R}\tron{\left(\phi^{T}_{1,\lambda}(t,x)\right)^{q^*+(1-q^*)\delta}\lambda^{-1}\;d \lambda}}d(t,x)\\
&\qquad+\int_{Q^3_{1,R}} \Phi_q(|u(t,x)|)\tron{\int_{0}^{R}[\xi_{1,\lambda}(x)]^{a}\lambda^{-1}\;d \lambda}d(t,x).
\end{aligned}
\end{equation}
Moreover, the function $\varphi$ is decreasing on $\left[0,\infty\right)$ and the function $\eta$ is decreasing on $\left[0,1\right]$. Hence, we may conclude from $t\lambda^{-1}\geq tR^{-1}$ and $|x|\lambda^{-1/\sigma}\geq |x|R^{-1/\sigma}$ that
\begin{equation}\label{thr1.2_eq10}
\phi^{X}_{1,\lambda}(t,x)\leq\phi^{X}_{1,R}(t,x) \text{ for } (t,x)\in Q^1_{1,R},\quad \phi^{T}_{1,\lambda}(t,x)\leq\phi^{T}_{1,R}(t,x) \text{ for } (t,x)\in Q^2_{R}.
\end{equation}
Denoting $\tilde{x}=x\lambda^{-1/\sigma}$ and $\tilde{t}=t\lambda^{-1}$, we deduce from \eqref{thr1.2_eq9} and \eqref{thr1.2_eq10} the following chain of inequalities:
\begin{align}
G_q(R) &\lesssim \int_{Q^1_{1,R}}\Phi_q(|u(t,x)|)[  \phi^{X}_{1,R}(t,x)]^{q^*+(1-q^*)\delta-\delta_1}\tron{\int_{1}^{\infty}\left(\varphi^{X}(\tilde{x})\right)^{\delta_1(n+2\kappa)}|\tilde{x}|^{-1}\;d|\tilde{x}|}\;d(t,x)\nonumber\\
&\quad+\int_{Q^2_{R}} \Phi_q(|u(t,x)|)[  \phi^{T}_{1,R}(t,x)]^{q^*+(1-q^*)\delta-\delta_1}\tron{\int_{1/2}^{1}\left(\eta^{T}(\tilde{t})\right)^{\nu\delta_1}\tilde{t}^{-1}\;d\tilde{t}}\;d(t,x)\nonumber\\
&\quad+\int_{Q^3_{1,R}}\Phi_q(|u(t,x)|)\tron{\int_{0}^{1}[\xi(\tilde{x})]^{a}|\tilde{x}|^{-1}\;d|\tilde{x}|}\;d(t,x)\nonumber\\
&\lesssim \int_{Q^1_{1,R}} \Phi_q(|u(t,x)|)[  \phi^{X}_{1,R}(t,x)]^{q^*+(1-q^*)\delta-\delta_1}d(t,x)\nonumber\\
&\quad+\int_{Q^2_{R}} \Phi_q(|u(t,x)|)[  \phi^{T}_{1,R}(t,x)]^{q^*+(1-q^*)\delta-\delta_1}d(t,x)+\int_{Q^3_{1,R}}\Phi_q(|u(t,x)|)d(t,x)\label{thr1.2_eq11}
\end{align}
for a sufficiently small positive constant $\delta_1$, which will be chosen later. An analogous argument implies
\begin{equation}\label{thr1.2_eq12}
\begin{aligned}
&G_p(R)\lesssim \int_{Q^1_{2,R}} \Phi_p(|v(t,x)|)\left(\phi^{X}_{2,R}\left(t,x\right)\right)^{p^*+(1-p^*)\delta-\delta_1}d(t,x)\\
&\quad+\int_{Q^2_{R}} \Phi_p(|v(t,x)|)\left(\phi^{T}_{2,R}\left(t,x\right)\right)^{p^*+(1-p^*)\delta-\delta_1}d(t,x)+\int_{Q^3_{2,R}} \Phi_p(|v(t,x)|)d(t,x).
\end{aligned}
\end{equation}
Since $\phi^X_{j,R}(t,x)\leq \big(\phi^{X}_{j,R}(t,x)\big)^{\frac{\nu}{\nu+2}}$, $\phi^T_{j,R}(t,x)\leq \big(\phi^{T}_{j,R}(t,x)\big)^{\frac{\nu}{\nu+2}}$ with $j=1,2$, it follows from \eqref{thr1.2_eq11} and \eqref{thr1.2_eq12} that
\[
\begin{aligned}
G_q(R) \lesssim \int_{Q_{R}}|u(t,x)|^{q^*} \mu_2(|u(t,x)|)[\phi_{1,R}(t,x)]^{\frac{\nu}{\nu+2}\tron{q^*+(1-q^*)\delta-\delta_1}} d(t,x), \\
G_p(R) \lesssim \int_{Q_{R}}|v(t,x)|^{p^*} \mu_1(|v(t,x)|)[\phi_{2,R}(t,x)]^{\frac{\nu}{\nu+2}\tron{p^*+(1-p^*)\delta-\delta_1}}d(t,x).
\end{aligned}
\]
By the choice $\delta_1<p^*+(1-p^*)\delta-1$ and 
\[ 
\nu\geq \max\left\{\frac{2}{p^*+(1-p^*)\delta-\delta_1-1},\frac{2}{q^*+(1-q^*)\delta-\delta_1-1}\right\},
\]
we obtain
\begin{equation}\label{thr1.2_eq13}
G_q(R) \lesssim \int_{Q_{R}}|u(t,x)|^{q^*} \mu_2(|u(t,x)|)\phi_{1,R}(t,x)d(t,x),
\end{equation}
\begin{equation}\label{thr1.2_eq14}
G_p(R) \lesssim \int_{Q_{R}}|v(t,x)|^{p^*} \mu_1(|v(t,x)|) \phi_{2,R}(t,x) d(t,x). 
\end{equation}
Let us define 
\begin{equation}
\begin{aligned}
    &\tilde{I}_{R}=\int_{Q_{R}}|v(t,x)|^{p^*} \mu_1(|v(t,x)|) \phi_{2,R}(t,x) d(t,x),\nonumber\\
    &\tilde{J}_{R}=\int_{Q_{R}}|u(t,x)|^{q^*} \mu_2(|u(t,x)|) \phi_{1,R}(t,x) d(t,x).\nonumber
\end{aligned}
\end{equation}
Since $\Phi_q\big(|u(t,x)|)\in L^1((0,\infty)\times \mathbb{R}^n\big)$, we have
\[
\lim_{R\to \infty}\tilde{J}_{R}=\lim_{R\to \infty}J_{R}=\int_{0}^{\infty}\int_{\mathbb{R}^n}\Phi_q(|u(t,x)|)d(t,x)<\infty.
\]
Hence, for $R>R_0$ large enough, it holds $J_{R}\leq \tilde{J}_{R}\leq 2J_{R}$. Because of $\tilde{I}_{R}\leq I_{R}$, we deduce from \eqref{thr1.2_eq13} and \eqref{thr1.2_eq14} that
\begin{equation}\label{thr1.2_eq15}
G_p(R)\lesssim I_{R}\quad \text{ and }\quad G_{q}(R)\lesssim J_{R}.
\end{equation}
\noindent Plugging \eqref{thr1.2_eq15} into \eqref{upperIR} and \eqref{upperJR} gives
\begin{align*}
 G_q(R)+\mathcal{D}(u_1)&\lesssim I_{R}+\mathcal{D}(u_1)\lesssim  R^{\frac{n}{2\sigma}} \Phi_p^{-1}\left(\frac{G_p^{\prime}(R)}{CR^{\frac{n}{2\sigma}}}\right),\\
G_p(R)+\mathcal{D}(v_1)&\lesssim J_{R}+\mathcal{D}(v_1)\lesssim  R^{\frac{n}{\sigma}-1}\Phi_q^{-1}\left(\frac{G_q^{\prime}(R)}{CR^{\frac{n}{\sigma}}}\right).
\end{align*}
Consequently, one achieves
\begin{align*}
    \Phi_p\left(\frac{G_q(R)+\mathcal{D}(u_1)}{CR^{\frac{n}{2\sigma}}}\right) \lesssim \frac{G_p^{\prime}(R)}{R^{\frac{n}{2\sigma}}}\text{ and }\Phi_q\left(\frac{G_p(R)+\mathcal{D}(v_1)}{CR^{\frac{n}{\sigma}-1}}\right) \lesssim \frac{G_q^{\prime}(R)}{R^{\frac{n}{\sigma}}}.
\end{align*}
Then, recalling the definition of the functions $\Phi_p$ and $\Phi_q$ we derive
\begin{align*}
    \left(\frac{G_q(R)+\mathcal{D}(u_1)}{CR^{\frac{n}{2\sigma}}}\right)^{p^*} \mu_1\left(\frac{G_q(R)+\mathcal{D}(u_1)}{CR^{\frac{n}{2\sigma}}}\right) &\lesssim \frac{G_p^{\prime}(R)}{R^{\frac{n}{2\sigma}}},  \\
    \left(\frac{G_p(R)+\mathcal{D}(v_1)}{CR^{\frac{n}{\sigma}-1}}\right)^{q^*} \mu_2\left(\frac{G_p(R)+\mathcal{D}(v_1)}{CR^{\frac{n}{\sigma}-1}}\right) &\lesssim \frac{G_q^{\prime}(R)}{R^{\frac{n}{\sigma}}}.
\end{align*}
These estimates imply
\begin{align*}
    R^{\frac{n}{2\sigma}-\frac{p^*n}{2\sigma}}\big(G_q(R)+\mathcal{D}(u_1)\big)^{p^*}\mu_1\left(\frac{G_q(R)+\mathcal{D}(u_1)}{CR^{\frac{n}{2\sigma}}}\right) &\lesssim G_p^{\prime}(R),   \\
    R^{\frac{n}{\sigma}-q^*\tron{\frac{n}{\sigma}-1}}\big(G_p(R)+\mathcal{D}(v_1)\big)^{q^*}\mu_2\left(\frac{G_p(R)+\mathcal{D}(v_1)}{CR^{\frac{n}{\sigma}-1}}\right) &\lesssim G_q^{\prime}(R)
\end{align*}
for all $R \geq R_0$. Due to the increasing property of the functions $\mu_1=\mu_1(s), \mu_2=\mu_2(s), G_p=$ $G_p(s)$ and $G_q=G_q(s)$, the following inequalities hold:
\begin{align*}
    R^{\frac{n}{2\sigma}-\frac{p^*n}{2\sigma}}\mu_1\left(C_0 R^{-\frac{n}{2\sigma}}\right)\big(G_q(R)+\mathcal{D}(u_1)\big)^{p^*} &\lesssim G_p^{\prime}(R), \\
    R^{\frac{n}{\sigma}-q^*\tron{\frac{n}{\sigma}-1}}\mu_2\big(C_0 R^{-\frac{n}{\sigma}+1}\big)\left(G_p(R)+\mathcal{D}(v_1)\right)^{q^*} &\lesssim G_q^{\prime}(R)
\end{align*}
for all $R\geq R_0$, where $C_0:=C^{-1}\min \left\{G_p\left(R_0\right), G_q\left(R_0\right)\right\}$. Denoting
$$
\tau_1(\rho):=\rho^{\frac{n}{2\sigma}-\frac{p^*n}{2\sigma}}\mu_1\left(C_0 \rho^{-\frac{n}{2\sigma}}\right), \quad \tau_2(\rho):=\rho^{\frac{n}{\sigma}-q^*\tron{\frac{n}{\sigma}-1}}\mu_2\left(C_0 \rho^{-\frac{n}{\sigma}+1}\right),
$$
we obtain the following system of ordinary differential inequalities for $\tau \geq R_0$ :
\begin{equation}\label{thr1.2_eq16}
G_p^{\prime}(R)\geq C_1 \tau_1(R)\left(G_q(R)\right)^{p^*}+C_1\tau_1(R)\big(\mathcal{D}(v_1)\big)^{p^*}, \\
\end{equation}
\begin{equation}\label{thr1.2_eq17}
G_q^{\prime}(R) \geq C_1 \tau_2(R)\left(G_p(R)\right)^{q^*}+C_1\tau_2(R)\big(\mathcal{D}(u_1)\big)^{q^*}.
\end{equation}
$\bullet$ {\bf Step 3:} For any $R\geq R_0$, after multiplying \eqref{thr1.2_eq16} by $G_q^{\prime}(\tau)$ and integrating by parts over $\left[R_0, R\right]$, one finds
\begin{align*}
& G_p(R) G_q^{\prime}(R)-G_p\left(R_0\right) G_q^{\prime}\left(R_0\right)-\int^R_{R_0} G_p(\rho) G_q^{\prime \prime}(\rho) d \rho \\
&\quad \geq \frac{C_1\theta}{p^*+1} \tau_1(R)\left(G_q(R)\right)^{p^*+1}-\frac{C_1\theta}{p^*+1} \tau_1\left(R_0\right)\left(G_q\left(R_0\right)\right)^{p^*+1} \\
&\qquad -\frac{C_1\theta}{p^*+1} \int_{R_0}^R \tau_1^{\prime}(\rho)\left(G_q(\rho)\right)^{p^*+1} d \rho+\int_{R_0}^{R}C_1\theta\big(\mathcal{D}(v_1)\big)^{p^*}\tau_1(\rho)G^{\prime}_q(\rho)\;d\rho
\end{align*}
for any $\theta\in [0,1]$. This leads to 
\begin{equation}\label{thr1.2_eq18}
\sum_{i=1}^3A_i\geq \sum_{i=1}^3{B_i}+\int_{R_0}^{R} C_1\theta\big(\mathcal{D}(v_1)\big)^{p^*}\tau_1(\rho)G^{\prime}_q(\rho)\;d\rho,
\end{equation}
where we use the notations
\begin{equation}\label{thr1.2_eq19}
    A_i:=G_{i,p}(R)G^{\prime}_{i,q}(R)-\int^R_{R_0} G_{i,p}(\rho) G^{\prime \prime}_{i,q}(\rho) d \rho 
\end{equation}
and 
\begin{equation}\label{thr1.2_eq20}
\begin{aligned}
B_i &:=\frac{C_1\theta}{p^*+1} \tau_1(R)\left(G_{i,q}(R)\right)^{p^*+1}+G_{i,p}\left(R_0\right) G^{\prime}_{i,q}\left(R_0\right) \\
&\qquad -\frac{C_1\theta}{p^*+1} \tau_1\left(R_0\right)\left(G_{i,q}\left(R_0\right)\right)^{p^*+1} -\frac{C_1\theta}{p^*+1} \int_{R_0}^R \tau_1^{\prime}(\rho)\left(G_{i,q}(\rho)\right)^{p^*+1} d \rho.
\end{aligned}
\end{equation}
Let us now give the estimates for $A_i,B_i$ with $i=1,2,3$. \medskip

\noindent \underline{{\bf Estimates for $A_1$ and $B_1$:}}
Thanks to the assumption \eqref{con2thr1.2}, it is clear to see that 
\[
\tau^{\prime}_1(\rho)=\rho^{-1}\tau_1(\rho)\tron{\frac{n}{2\sigma}-\frac{p^*n}{2\sigma}-\frac{n}{2\sigma}\frac{\mu^{\prime}_1(C_0\rho^{-\frac{n}{2\sigma}})}{\mu_1(C_0\rho^{-\frac{n}{2\sigma}})}C_0\rho^{-\frac{n}{2\sigma}}}\leq 0
\]
for all $\rho\geq R_0$, provided that $R_0$ is large enough. Choosing a small enough constant $\theta$ such that
\[
G_{1,p}\left(R_0\right) G^{\prime}_{1,q}\left(R_0\right)-\frac{C_1\theta}{p^*+1} \tau_1\left(R_0\right)\left(G_{1,q}\left(R_0\right)\right)^{p^*+1}>0,
\]
we arrive at 
\begin{equation}\label{esB1}
B_1\geq \frac{C_1\theta}{p^*+1} \tau_1(R)\left(G_{1,q}(R)\right)^{p^*+1}.
\end{equation}
In order to estimate $A_1$, we divide the range $[R_0,R]$ into two intervals $[R_0,R/2^{\sigma}]$ and $[R/2^{\sigma},R]$. We show $G^{\prime \prime}_{1,q}(\rho)\geq 0$ for any $\rho \in [R_0,R/2^{\sigma}]$ by indicating that the function
\begin{equation}\label{3.34}
[\varphi^{X}_{1,\rho}(x)]^{(n+2\kappa)(q^*+(1-q^*)\delta)}\rho^{-1} \text{ is an increasing function with respect to $\rho$}.
\end{equation}
Indeed, taking the derivative of this function we need to prove
\begin{equation}\label{thr1.2_eq21}
b+\tron{\frac{|x|}{\rho^{1/\sigma}}-1}\tron{\frac{|x|}{\rho^{1/\sigma}}\tron{b-\frac{4}{\sigma}}-b}\leq 0, 
\end{equation}
where
\[
b:=\frac{4}{(n+2\kappa)(q^*+(1-q^*)\delta)}.
\]
Since $(t,x)\in Q^1_{1,R}$, then $|x|/R^{1/\sigma}\geq 1$, we have
\[
\rho\leq\frac{R}{2^{\sigma}}\Longrightarrow \frac{|x|}{\rho^{1/\sigma}}\geq 2\frac{|x|}{R^{1/\sigma}}\geq 2.
\]
Since $n>\sigma$, it holds
\[
b-\frac{4}{\sigma}<\frac{4}{n+2\kappa}-\frac{4}{\sigma}<0.
\]
Then, LHS of \eqref{thr1.2_eq21} $<b+(2-1)(-b)\leq 0$, which gives \eqref{3.34}. So we conclude that for $\rho\in \left[R_0,R/2^\sigma\right]$ the function $g_{1,q}(\rho)/\rho$ is increasing. It follows that
\begin{equation}\label{thr1.2_eq22}
\int^{R/2^{\sigma}}_{R_0}G_{1,p}(\rho)G^{\prime \prime}_{1,q}(\rho)\;d\rho\geq 0.
\end{equation}
On the interval $[R/2^{\sigma},R]$, we get
\begin{equation}\label{thr1.2_eq23}
-\di\int^R_{R/2^{\sigma}} G_{1,p}(\rho) G^{\prime \prime}_{1,q}(\rho) d\rho\leq\int^R_{R/2^{\sigma}} \frac{G_{1,p}(\rho)G^{\prime}_{1,q}(\rho)}{\rho} d\rho,
\end{equation}
since
\[
G^{\prime \prime}_{1,q}(\rho)=\frac{g^{\prime}_{1,q}(\rho)-G^{\prime}_{1,q}(\rho)}{\rho}\leq -\frac{G^{\prime}_{1,q}(\rho)}{\rho}.
\]
The function $G^{\prime}_{1,q}(\rho)\rho= g_{1,q}(\rho)$ is increasing, so
\[
\frac{G^{\prime}_{1,q}(\rho)}{\rho}\leq \frac{G^{\prime}_{1,q}(R)R}{\rho^2}\leq \frac{2^{2\sigma}G^{\prime}_{1,q}(R)}{R}\text{ for $\rho\in \left[R/2^{\sigma},R\right]$}.
\]
Thus, one finds
\begin{align}
\int^R_{R/2^{\sigma}}\frac{G_{1,p}(\rho)G^{\prime}_{1,q}(\rho)}{\rho}d\rho\leq G_{1,p}(R)\int^{R}_{R/2^{\sigma}}\frac{G^{\prime}_{1,q}(\rho)}{\rho}\;d \rho&\leq \frac{(2^{2\sigma}-2^{\sigma})G_{1,p}(R)G^{\prime}_{1,q}(R)}{R}R\nonumber\\
&\lesssim G_{1,p}(R)G^{\prime}_{1,q}(R).\label{thr1.2_eq24}
\end{align}
Combining \eqref{thr1.2_eq23} and \eqref{thr1.2_eq24} we obtain 
\begin{equation}\label{thr1.2_eq25}
    -\di\int^R_{R/2^{\sigma}} G_{1,p}(\rho) G^{\prime \prime}_{1,q}(\rho) d\rho\lesssim G_{1,p}(R)G^{\prime}_{1,q}(R).
\end{equation}
Hence, from \eqref{thr1.2_eq22} and \eqref{thr1.2_eq25} one deduces
\begin{equation}
A_1\lesssim G_{1,p}(R)G^{\prime}_{1,q}(R).\label{esA1}
\end{equation}
\noindent \underline{{\bf Estimates for $A_2$ and $B_2$:}}
The estimate for $B_2$ is quite similar to \eqref{esB1}, namely,
\begin{equation}\label{esB2}
B_2\gtrsim \tau_1(R)\left(G^{2}_q(R)\right)^{p^*+1},
\end{equation}
which can be derived in the same way. To estimate $A_2$, for $\rho\in [R_0,R/2^{\sigma}]$ we have
\[
\frac{t}{\rho}\geq 2^{\sigma}\frac{t}{R}\geq \frac{2^{\sigma}}{2}>1
\]
since $\sigma>1$, which leads to
\[
\eta^{T}_{\rho}(t)=0 \text{ and }\int_{Q^2_{R}}\Phi_q(|u(t,x)|)[\phi^{T}_{1,\rho}(t,x)]^{q^*+(1-q^*)\delta}\;d(t,x)=0.
\]
Hence, one arrives at
\[
G^{\prime \prime}_{1,q}(\rho)=0\text{ and }\int^{R/2^{\sigma}}_{R_0}G_{2,p}(\rho)G^{\prime \prime}_{1,q}(\rho)\;d\rho=0.
\]
The range $\rho \in [R/2^{\sigma},R]$ can be also handled by the similar argument to estimate $A_1$. Summarizing, we conclude that
\begin{equation}
A_2\lesssim G_{2,p}(R)G^{\prime}_{1,q}(R). \label{esA2}
\end{equation}
\noindent \underline{{\bf Estimates for $A_3$ and $B_3$:}} Since 
\[
\rho^{1+\frac{a}{\sigma}}G^{\prime}_{3,q}(\rho)=\int_{Q^3_{1,R}}\Phi_q(|u(t,x)|)|x|^{a}\;d(t,x)
\]
for any $\rho \in [R_0,R]$, we have
\[
0=\tron{\rho^{1+\frac{a}{\sigma}}G^{\prime}_{3,q}(\rho)}^{\prime}=G^{\prime \prime}_{3,q}(\rho)\rho^{1+\frac{a}{\sigma}}+\left(1+\frac{a}{\sigma}\right)\rho^{\frac{a}{\sigma}}G^{\prime}_{3,q}(\rho).
\]
Hence, it holds
\[
  -G^{\prime \prime}_{3,q}(\rho)=\tron{1+\frac{a}{\sigma}}\frac{G^{\prime}_{3,q}(\rho)}{\rho}.  
\]
Using the relation $G^{\prime}_{3,q}(\rho)\rho=g_{3,q}(\rho)$, we obtain
\begin{align}
-\int_{R_0}^{R}G^{\prime \prime}_{3,q}(\rho)G_{3,p}(\rho)\;d\rho&=\tron{1+\frac{a}{\sigma}}\int_{R_0}^{R}\frac{G^{\prime}_{3,q}(\rho)G_{3,p}(\rho)}{\rho}\;d\rho\nonumber\\
&=\tron{1+\frac{a}{\sigma}}\int_{R_0}^{R}\frac{g_{3,q}(\rho)G_{3,p}(\rho)}{\rho^{2}}\;d\rho.\label{thr1.2_eq26}
\end{align}
Let us estimate the right-hand side of \eqref{thr1.2_eq26} in the following way:
\begin{align*}
G_{3,p}(\rho)&=\int_{0}^{\rho}g_{3,q}(\lambda)\lambda^{-1}\;d\lambda=\int_{Q^3_{1,R}}\tron{\Phi_p(|u(t,x)|)\int_{0}^{\rho}[\xi_{1,\lambda}(x)]^{a}\lambda^{-1}\;d\lambda}\;d(t,x)\\
&\geq \int_{Q^3_{1,R}}\Phi_p(|u(t,x)|)[\xi_{1,\rho}(x)]^{a}\;d(t,x)\int_{\rho e^{-\frac{\sigma}{a}}}^{\rho}\lambda^{-1}\;d\lambda\\
&\geq \frac{\sigma}{a}\int_{Q^3_{1,R}}\Phi_p(|u(t,x)|)[\xi_{1,\rho}(x)]^{a}\;d(t,x)=\frac{\sigma g_{3,p}(\rho)}{a}.
\end{align*}
This leads to
\[
\tron{\frac{G_{3,p}(\rho)}{\rho^{a/\sigma}}}^{\prime}=\frac{[G_{3,p}]^{\prime}(\rho)\rho^{a/\sigma}-(a/\sigma)\rho^{a/\sigma-1}G_{3,p}(\rho)}{\rho^{2a/\sigma}}=\frac{g_{3,p}(\rho)-(a/\sigma)G_{3,p}(\rho)}{\rho^{1+a/\sigma}}\leq 0,
\]
which implies
\[
\begin{aligned}
\big(g_{3,q}(\rho)G_{3,p}(\rho)\big)^{\prime}&=\int_{Q^3_R}\Phi_q(|u(t,x)|)|x|^{a}\;d(t,x)\tron{\frac{G_{3,p}(\rho)}{\rho^{a/\sigma}}}^{\prime}\leq 0.
\end{aligned}
\]
From this, one realizes that $g_{3,q}(\rho)G_{3,p}(\rho)$ is a non-increasing function. Therefore, from \eqref{thr1.2_eq26} it holds
\begin{align*}
-\int_{R_0}^{R}G^{\prime \prime}_{3,q}(\rho)G_{3,p}(\rho)d\rho&\leq \tron{1+\frac{a}{\sigma}}g_{3,q}(R_0)G_{3,p}(R_0)\int_{R_0}^{R}\frac{1}{\rho^2}d\rho\\
&\leq \tron{1+\frac{a}{\sigma}}g_{3,q}(R_0)G_{3,p}(R_0)\tron{\frac{1}{R_0}-\frac{1}{R}}\\
&\leq \tron{1+\frac{a}{\sigma}}G^{\prime}_{3,q}(R_0)G_{3,p}(R_0).
\end{align*}
This gives
\begin{align}
A_3&=G_{3,p}(R)G^{\prime}_{3,q}(R)-\int_{R_0}^{R}G^{\prime \prime}_{3,q}(\rho)G_{3,p}(\rho)d\rho\nonumber\\
&\leq G_{3,p}(R)G^{\prime}_{3,q}(R)+\tron{1+\frac{a}{\sigma}}G^{\prime}_{3,q}(R_0)G_{3,p}(R_0).\label{esA3}
\end{align}
Next, we notice that the relation $\tau_1^{\prime}(\rho)\leq 0$ for all $\rho\geq R_0$ leads to
\[
\begin{aligned}
-\int_{R_0}^{R}\tau^{\prime}_1(\rho)(G^{3}_q(\rho))^{p^*+1}\;d\rho &\geq [G^{3}_q(R_0)]^{p^*+1}\int_{R_0}^{R}-\tau^{\prime}_1(\rho)\;d\rho  \\
&\geq \tau_1(R_0)[G^{3}_q(R_0)]^{p^*+1}-\tau_1(R)[G^{3}_q(R_0)]^{p^*+1},
\end{aligned}
\]
which yields
\begin{equation}\label{esB3}
B_3\geq \frac{C_1\theta}{p^*+1} \tau_1(R)\left(G_{3,q}(R)\right)^{p^*+1}+G_{3,p}\left(R_0\right) G^{\prime}_{3,q}\left(R_0\right)-\frac{C_1\theta}{p^*+1}\tau_1(R)[G_{3,q}(R_0)]^{p^*+1}.
\end{equation}
\noindent $\bullet$ {\bf Step 4:}
Summarizing all estimates \eqref{esA1}, \eqref{esB1}, \eqref{esA2}, \eqref{esB2}, \eqref{esA3}, \eqref{esB3}, we obtain
\begin{align}
G^{\prime}_q(R)G_p(R)\geq C_2\tau_1(R)(G_q(R))^{p^*+1}&+C_1\big(\mathcal{D}(v_1)\big)^{p^*}\int_{R_0}^{R}\tau_1(\rho)G^{\prime}_q(\rho)\;d\rho\nonumber\\
&-C_3\tau_1(R)[G_{3,q}(R_0)]^{p^*+1}-\frac{a}{\sigma}G^{\prime}_{3,q}(R_0)G_{3,p}(R_0)\nonumber\\
:=C_2\tau_1(R)(G_q(R))^{p^*+1}&+E_1-E_2-E_3.\label{thr1.2_eq27}
\end{align}
To estimate the term $E_1$, one observes that the function $\mu_1\tron{C_0\rho^{-n/2\sigma}}\rho$ is non-decreasing. Indeed, due to the condition \ref{con2thr1.2}, it holds
\[
\begin{aligned}
    \tron{\mu_1\tron{C_0\rho^{-\frac{n}{2\sigma}}}\rho}^{\prime}&=\tron{C_0\tron{-\frac{n}{2\sigma}}\rho^{-\frac{n}{2\sigma}}\mu^{\prime}_1\tron{C_0\rho^{-\frac{n}{2\sigma}}}+\mu_1\tron{C_0\rho^{-\frac{n}{2\sigma}}}}\\
    &=\mu_1\tron{C_0\rho^{-\frac{n}{2\sigma}}}\tron{-C_0\tron{\frac{n}{2\sigma}}\frac{\rho^{-\frac{n}{2\sigma}}\mu^{\prime}_1\tron{C_0\rho^{-\frac{n}{2\sigma}}}}{\mu_1\tron{C_0\rho^{-\frac{n}{2\sigma}}}}+1}\geq 0
\end{aligned}
\]
for any $\rho\in [R_0,R]$. This gives
\begin{align}
\int_{R_0}^{R}\tau_1(\rho)G^{\prime}_q(\rho)\;d\rho&=\int_{R_0}^{R}\frac{\mu_1\tron{C_0\rho^{-\frac{n}{2\sigma}}}g_q(\rho)}{\rho^{1+\frac{p^*n}{2\sigma}-\frac{n}{2\sigma}}}\;d\rho=\int_{R_0}^{R}\frac{\mu_1\tron{C_0\rho^{-\frac{n}{2\sigma}}}\rho^{1+a/\sigma}g_p(\rho)}{\rho^{2+a/\sigma+\frac{n}{2\sigma}(p^*-1)}}\;d\rho \nonumber \\
&\geq \mu_1\tron{C_0R_0^{-\frac{n}{2\sigma}}}R_0^{1+a/\sigma}g_p(R_0)\int_{R_0}^{R}\frac{1}{\rho^{2+a/\sigma+\frac{n}{2\sigma}(p^*-1)}}\;d\rho, \label{thr1.2_eq28}
\end{align}
which implies
\[
E_1=C_1\big(\mathcal{D}(v_1)\big)^{q^*}\int_{R_0}^{R}\tau_1(\rho)G^{\prime}_q(\rho)\;d\rho\geq \big(\mathcal{D}(v_1)\big)^{q^*}\underbrace{C\tron{\mu_1\tron{C_0R_0^{-n/2\sigma}},g_p(R_0),R_0,C_1}}_{C_4}>0
\]
for $R\geq 2R_0$. On the other hand, since $\tau_1(R)\to 0$ as $R\to \infty$, there exists $R_1>R_0$ such that
\[
E_2=C_3\tau_1(R)[G_{3,q}(R_0)]^{p^*+1}\leq \frac{\big(\mathcal{D}(v_1)\big)^{q^*}C_4}{2}
\]
for any $R\geq R_1$. We can also choose the parameter $a$ small enough fulfilling
\[
E_3= \frac{a}{\sigma}G^{\prime}_{3,q}(R_0)G_{3,q}(R_0)\leq \frac{\big(\mathcal{D}(v_1)\big)^{q^*}C_4}{2}.
\]
Those estimates imply 
\[
G_p(R) G_q^{\prime}(R) \geq C_2 \tau_1(R)\left(G_q(R)\right)^{p^*+1}
\]
for any $R\geq R_2:=\max\{2R_0,R_1\}$. Consequently, it holds
\begin{equation}\label{thr1.2_eq29}
G_p(R) \gtrsim \frac{C_2 \tau_1(R)\left(G_q(R)\right)^{p^*+1}}{G_q^{\prime}(R)}.
\end{equation}
Plugging \eqref{thr1.2_eq29} into \eqref{thr1.2_eq17}, one gets
$$
G_q^{\prime}(R) \gtrsim \frac{C_2 \tau_2(R)\left(\tau_1(R)\right)^{q^*}\left(G_q(R)\right)^{q^*\left(p^*+1\right)}}{\left(G_q^{\prime}(R)\right)^{q^*}},
$$
which is equivalent to
\begin{align*}
G_q^{\prime}(R) \gtrsim & \left(\tau_2(R)\right)^{\frac{1}{q^*+1}}\left(\tau_1(R)\right)^{\frac{q^*}{q^*+1}}\left(G_q(R)\right)^{\frac{q^*\left(p^*+1\right)}{q^*+1}} \nonumber\\
& =\frac{1}{R}\left(\mu_1\left(C_0 R^{-\frac{n}{2\sigma}}\right)\right)^{\frac{q^*}{q^*+1}}\left(\mu_2\left(C_0 R^{-\frac{n}{\sigma}+1}\right)\right)^{\frac{1}{q^*+1}}\left(G_q(R)\right)^{\frac{q^*\left(p^*+1\right)}{q^*+1}},
\end{align*}
where we have utilized the relation
\[
\tron{\frac{n}{2\sigma}-\frac{np^*}{2\sigma}}\frac{q^*}{q^*+1}+\tron{\frac{n}{\sigma}-q^*\tron{\frac{n}{\sigma}-1}}\frac{1}{q^*+1}=-1.
\]
Therefore, we arrive at
\begin{equation}\label{thr1.2_eq30}
\frac{1}{R}\left(\mu_1\left(C_0 R^{-\frac{n}{\sigma}+1}\right)\right)^{\frac{q^*}{q^*+1}}\left(\mu_2\left(C_0 R^{-\frac{n}{2\sigma}}\right)\right)^{\frac{1}{q^*+1}} \lesssim \frac{G_q^{\prime}(R)}{\left(G_q(R)\right)^{\frac{q^*\left(p^*+1\right)}{q^*+1}}}
\end{equation}
for any $R\geq R_2$. Integrating both sides of \eqref{thr1.2_eq30} over $[R_2,R]$ leads to
\[
\begin{aligned}
& \int_{R_2}^{R} \frac{1}{r}\left(\mu_1\left(C_0 r^{-\frac{n}{\sigma}+1}\right)\right)^{\frac{q^*}{q^*+1}}\left(\mu_2\left(C_0 r^{-\frac{n}{2\sigma}}\right)\right)^{\frac{1}{q^*+1}} d r \\
&\qquad \lesssim \int_{R_2}^{R} \frac{G_q^{\prime}(r)}{\left(G_q(r)\right)^{\frac{q^*\left(p^*+1\right)}{q^*+1}}} d r=-\left.\frac{q^*+1}{p^* q^*-1}\left(G_q(r)\right)^{-\frac{p^*q^*-1}{q^*+1}}\right|_{r=R_2}^{r=R} \lesssim \big(G_q\left(R_2\right)\big)^{-\frac{p^*q^*-1}{q^*+1}}.
\end{aligned}
\]
After letting $R\to \infty$, we obtain
\[
\int_{R_2}^{\infty} \frac{1}{r}\left(\mu_1\left(C_0 r^{-\frac{n}{\sigma}+1}\right)\right)^{\frac{q^*}{q^*+1}}\left(\mu_2\left(C_0 r^{-\frac{n}{2\sigma}}\right)\right)^{\frac{1}{q^*+1}} d r\lesssim \big(G_q\left(R_2\right)\big)^{-\frac{p^*q^*-1}{q^*+1}}.
\]
Since $n/2\sigma>n/\sigma-1$ and $\mu_1$ is an increasing function, one gets
\[
\int_{R_2}^{\infty} \frac{1}{r}\left(\mu_1\left(C_0 r^{-\frac{n}{2\sigma}}\right)\right)^{\frac{q^*}{q^*+1}}\left(\mu_2\left(C_0 r^{-\frac{n}{2\sigma}}\right)\right)^{\frac{1}{q^*+1}} d r\lesssim \big(G_q\left(R_2\right)\big)^{-\frac{p^*q^*-1}{q^*+1}}.
\]
Finally, carrying out the change of variables $s=C_0 r^{-n/2\sigma}$ we claim that
$$
\int_0^{C_0 R_2^{-n/2\sigma}}\frac{1}{s}\big(\mu_1(s)\big)^{\frac{q^*}{q^*+1}}\big(\mu_2(s)\big)^{\frac{1}{q^*+1}} d s \lesssim \big(G_q\left(R_2\right)\big)^{-\frac{p^*q^*-1}{q^*+1}},
$$
which a contradiction to the assumption \eqref{con3thr1.2}. For this reason, we have completed the proof for the case $p^*\leq p_{\rm c}$ and $q^*\geq q_{\rm c}$. 

{\bf Next, let us consider the case $p^*\geq p_{\rm c}$ and $q^*\leq q_{\rm c}$}. 
This case can be proved by interchanging the roles of $p^*$ and $q^*$ in the proof of the previous case and using the following identity:
\[
\frac{p^*}{p^*+1}\tron{\frac{n}{\sigma}-q^*\tron{\frac{n}{\sigma}-1}}+\frac{1}{p^*+1}\tron{\frac{n}{2\sigma}-p^*\frac{n}{2\sigma}}=-1.
\]
However, we have to notice that there is a small difference between the case $\sigma$ is an integer and is fractional. More precisely, this arises when we estimate the term $A_1$. Since the roles of $p^*$ and $q^*$ are switched, the domain $Q^1_{1,R}$ would be switched to $Q^1_{2,R}$. For this reason, the estimate \eqref{thr1.2_eq21} now becomes
\[
b+\tron{\frac{|x|}{\rho^{1/2\sigma}}-1}\tron{\frac{|x|}{\rho^{1/2\sigma}}\tron{b-\frac{2}{\sigma}}-b}\leq 0,
\]
If $\sigma$ is an integer, then it is  obvious to see that $\kappa=\sigma$ and $b-2/\sigma<0$. If $\sigma$ is fractional, then $\kappa=\sigma-\flo{\sigma}$. So we need the condition $n\geq 2\flo{\sigma}$ to have
\[
b-\frac{2}{\sigma}<\frac{4}{n+2\kappa}-\frac{2}{\sigma}\leq 0.
\]
Another point worth emphasizing is that we require $\tau_2(R) \to 0$ as $R \to \infty$. To verify this, it is necessary to show
\[
\frac{n}{\sigma} - q^* \left( \frac{n}{\sigma} - 1 \right) < 0, \quad \text{ that is, }\quad q^* > 1 + \frac{\sigma}{n - \sigma},
\]
which holds naturally since $(p^*, q^*)$ lies on the curve \eqref{critcurve}. This completes our proof.
\begin{remark}
{\rm
The proof of Theorem \ref{thr1.2} represents a nontrivial improvement over the methods developed in \cite{TangDaoCung2026} and \cite{AnhRei2021}, which apply only when $\sigma$ is an integer. When $\sigma$ is fractional, compactly supported cutoff test functions can be no longer applicable since supp$(-\Delta)^\sigma \phi$ is, in general, bigger than supp$\phi$ for any $\phi \in \mathcal{C}_0^\infty$. Consequently, the singular behavior of $G_q(s)$ and $G_p(s)$ both near the origin and at infinity gives some difficulties. Our main contributions are based on constructing suitable test functions and decomposing the integral domain into several regions $Q^1_{1,R},\ Q^1_{2,R},\ Q^2_{R},\ Q^3_{1,R}$ and $Q^3_{2,R}$ in order to handle this situation.
}    
\end{remark}
\subsection{A further remark on the symmetric models}\label{subsec4.3}
Finally, let us discuss the blow-up result in the critical case for the following weakly coupled system of semilinear structurally damped $\sigma$-evolution equations:
\begin{equation}\label{sys_symmetric}
\left\{
\begin{aligned}
    &u_{tt}+(-\Delta)^{\sigma}u+(-\Delta)^{\delta}u_t&&=|v|^{p^*}\mu_1(|v|), &&x\in \mathbb{R}^n,\, t\geq 0,\\
    &v_{tt}+(-\Delta)^{\sigma}v+(-\Delta)^{\delta}v_t&&=|u|^{q^*}\mu_2(|u|), &&x\in \mathbb{R}^n,\, t\geq 0,\\
    &(u,u_t,v,v_t)(0,x)&&=(0,u_1,0,v_1)(x), &&x\in \mathbb{R}^n,
\end{aligned}
\right.
\end{equation}
where $\sigma\ge 1$, $\delta\in [0,\sigma]$ and $(p^*,q^*)$ belongs to the critical curve (see more \cite{Anh2020,Abbicco2025}) as follows:
\begin{equation}\label{crit3}\tag{$**$}
\frac{\max\{p,q\}+1}{pq-1}-\frac{n-\min\{\sigma,2\delta\}}{2\sigma}= 0.
\end{equation}
The existence of global (in time) Sobolev solutions to \eqref{sys_symmetric} can be obtained by adapting the method used in \cite{TangDaoCung2026}. However, the main difficulty comes from establishing a blow-up result for \eqref{sys_symmetric} when both $\sigma$ and $\delta$ are fractional. Let us provide a positive answer to this question via the following theorem.
\begin{theorem} \label{thr_blowup}
   Let $\delta \in [0, \sigma]$, $n > \min\{\sigma, 2\delta\}$ and $(p^*,q^*)$ satisfying \eqref{crit3}. We require the condition $n \geq 2\lfloor\sigma\rfloor$ when $\sigma$ is fractional and $\delta \in [0, \sigma/2]$, as well as $n < 2\sigma$ when $\delta \in (\sigma/2, \sigma]$. Assume that the initial data $u_1, v_1 \in L^1$ fulfill the relations
    \begin{equation*}\label{con1thr_blowup}
        \int_{\mathbb{R}^n}u_1(x)dx>0\quad \text{and}\quad \int_{\mathbb{R}^n}v_1(x)dx>0.
    \end{equation*}
    Moreover, we suppose the following assumptions of moduli of continuity:
    \begin{equation*}\label{con2thr_blowup}
        s^{k}\mu^{(k)}_{j}(s)=\mathcal{O}(\mu_j(s))\text{ as }s\to +0\text{ with }j,k=1,2
    \end{equation*}
    and
\[
I_{\mu_1,\mu_2} :=
\begin{cases}
    \vspace{0.2cm}\displaystyle\int_{0}^{c} \frac{1}{s} (\mu_1(s))^{\frac{1}{1+p^*}} (\mu_2(s))^{\frac{p^*}{1+p^*}} \, ds &\text{ if } p^*\ge q^*, \\
    \displaystyle\int_{0}^{c} \frac{1}{s} (\mu_1(s))^{\frac{q^*}{1+q^*}} (\mu_2(s))^{\frac{1}{1+q^*}} \, ds &\text{ if } q^*\ge p^*, \\
\end{cases}
\]
    where $c>0$ is a suitable small constant. Then, there is no global (in time) Sobolev solution to \eqref{sys_symmetric}.
\end{theorem}

\begin{proof}
    Let us sketch briefly the establishment of this result as follows: For the case where $\sigma$ is an integer, the proof of Theorem \ref{con1thr_blowup} can be followed from similar arguments to that of Theorem 1.1 in \cite{TangDaoCung2026} by introducing the following test functions instead: 
    \begin{align*}
    &\phi_{R}(t,x)=\left[\varphi\big(R^{-1/2\sigma}x\big)\right]^{n+2\sigma}\left[\eta\big(R^{-(\sigma-\delta)/\sigma}t)\right]^{\nu},&&\text{ if }\delta\in [0,\sigma/2],\\
    &\phi_{R}(t,x)=\left[\varphi\big(R^{-1/\sigma}x\big)\right]^{n+2\sigma}\left[\eta\big(R^{-1}t)\right]^{\nu},&&\text{ if }\delta\in (\sigma/2,\sigma].
    \end{align*}
    However, when $\sigma$ is fractional, the emergence of the domains 
    \begin{align*}
        &Q^3_{R} := \Big\{(t,x) : |x| \leq R^{1/2\sigma}, t \leq R^{(\sigma-\delta)/\sigma}\Big\},&&\text{ if }\delta\in [0,\sigma/2],\\
        &Q^3_{R} := \big\{(t,x) : |x| \leq R^{1/\sigma}, t \leq R\big\},&&\text{ if }\delta\in (\sigma/2,\sigma],
    \end{align*}
    pose remarkable challenges to the same method as in \cite{TangDaoCung2026} (this explains why it leaves an open problem there). For this reason, we adapt some arguments in the proof of Theorem \ref{thr1.2}, more precisely, the treatment of $A_3$ and $B_3$ in the region near the origin. The remainder of the proof is analogous to that of Theorem \ref{thr1.2}, hence, this completes our proof.
\end{proof}

\section*{Acknowledgments}
This research is funded by Vietnam National Foundation for Science and Technology Development (NAFOSTED) under grant number 101.02-2025.23.

\appendix
\section{Some auxiliary estimates for test functions}

\begin{lemma}\label{lem_Laplace_integer}
Let $q>0$ and $\varphi:\mathbb{R}^n\to \mathbb{R}$ is defined as follows:
\begin{equation}\label{varphi_defined}
\varphi(x):=\left\{
\begin{aligned}
&1&&\text{ for }|x|\leq 1,\\
&\big(1+(|x|-1)^{4}\big)^{1/4}&&\text{ for }|x|\geq 1.
\end{aligned}
\right.
\end{equation}
Then, the following estimates hold for any multi-index $\alpha$ satisfying $|\alpha|\geq 1$:
\[
\left|\partial^\alpha_x\varphi(x)^{-q}\right|
\begin{cases}
    \lesssim 1 &\text{ for }|x|\leq 2,\,|x|\neq 1,\\
    \sim \varphi(x)^{-q-|\alpha|} &\text{ for }|x|\geq 2.
\end{cases}
\]
\end{lemma}

\begin{proof}
The conclusion of this lemma is obvious for $|x|<1$. For $|x|>1$, we may follow the proof of Lemma 2.1 in \cite{Anhblowup} with minor modifications to conclude the desired estimate. 
\end{proof}

\begin{lemma}\label{lem_Laplace_fractional}
    Let $\varphi(x)$ be defined as in \eqref{varphi_defined}. Let $s\in (0,1)$, $m\in \mathbb{N}$, $\gamma=m+s$ and $q+2m>n$. Then the following estimate holds for any $x\in \mathbb{R}^n$ and $|x|\neq 1$:
    \[
    |(-\Delta)^{\gamma}\varphi(x)^{-q}|\lesssim \varphi(x)^{-n-2s}.
    \]
\end{lemma}

Before giving the proof of Lemma \ref{lem_Laplace_fractional}, we want to verify the following result.

\begin{lemma}\label{lem_Laplace_(0,1)}
Let $\varphi(x)$ be defined as in \eqref{varphi_defined}, $s \in (0,1)$ and $q>n$. Assume that $\phi\in W^{2,\infty}$ satisfies 
\[
|\phi(x)|
\begin{cases}
    \lesssim 1 &\text{ for }|x|\leq 2,\,|x|\neq 1,\\
    \sim \varphi(x)^{-q} &\text{ for }|x|\geq 2,
\end{cases}
\]
and $\left|\partial^2_x\phi(x)\right|\lesssim \varphi(x)^{-q-2}$ for $|x|>1$. Then, the following estimate holds for any $x\in \mathbb{R}^n$:
\begin{equation}\label{10}
\left|(-\Delta)^s\phi(x)\right|\lesssim \varphi(x)^{-n-2s}.
\end{equation}
\end{lemma}

\begin{proof}
Since $\phi,\,\partial_x^2\phi \in L^\infty$, we can remove the principal value of the integral at the origin and conclude that
$$(-\Delta)^s\phi(x)=-\frac{C_{n,s}}{2} \int_{\mathbb{R}^n}\frac{\phi(x+y)+\phi(x-y)-2\phi(x)}{|y|^{n+ 2s}}\,dy.$$
To establish the desired estimate, we divide the analysis into two cases: $|x|\leq 2$ with $x\neq 1$, and $|x|\geq 2$. The former case can be treated via an explicit computation as follows: 
\begin{align}
|(-\Delta)^s\phi(x)| &\lesssim \int_{|y|\leq1}\frac{|\phi(x+y) + \phi(x-y) - 2\phi(x)|}{|y|^{n+2s}}\,dy \nonumber\\
&\quad + \int_{|y|\geq1}\frac{|\phi(x+y) + \phi(x-y) - 2\phi(x)|}{|y|^{n+2s}}\,dy \nonumber\\
&\lesssim 1. \label{case1eq1}
\end{align}
Meanwhile, the latter case is followed by adapting some arguments used in the proof of Lemma 2.3 in \cite{Anhblowup} with minor modifications, so that we may omit its details for brevity.
\end{proof}

\begin{proof}[Proof of Lemma \ref{lem_Laplace_fractional}]
\noindent The first case $m=0$ was proved in Lemma \ref{lem_Laplace_(0,1)}. Let us turn to the second case $m\geq 1$. Namely, Lemma \ref{lem_Laplace_integer} implies that
\[ |(-\Delta)^{m}\varphi(x)^{-q}|
\begin{cases}
    \lesssim 1 &\text{ for }|x|\leq 2,\,|x|\neq 1,\\
    \sim |\partial^{2m}_x\varphi(x)^{-q}|\sim\varphi(x)^{-q-2m} &\text{ for }|x|\geq 2.
\end{cases}
\]
By choosing $\phi(x)=(-\Delta)^{m}\varphi(x)^{-q}$, then $\phi(x)$ satisfies 
\[|\phi(x)|
\begin{cases}
    \lesssim 1&\text{ for }|x|\leq 2,\,|x|\neq 1,\\
    \sim \varphi(x)^{-q-2m}&\text{ for }|x|\geq 2,\\
\end{cases}
\]
and
$$ |\partial^2_x\phi(x)|\sim \left|\partial^{2m+2}_x\varphi(x)^{-q}\right|\lesssim \varphi(x)^{-q-2m-2} \text{ for } x\in \mathbb{R}^n \text{ and }|x|\neq 1, $$
where Lemma \ref{lem_Laplace_integer} is used again in the last estimate. Applying Lemma \ref{lem_Laplace_(0,1)} we get
\[
|(-\Delta)^s\phi(x)|\lesssim \varphi(x)^{-n-2s}.
\]
for $x\in \mathbb{R}^n$ and $|x|\neq 1$. Due to the fact that
$$ (-\Delta)^{s}\phi(x)=(-\Delta)^{s}\big((-\Delta)^{m}\varphi(x)^{-q}\big)=(-\Delta)^{\gamma}\varphi(x)^{-q}, $$
we complete the proof of Lemma \ref{lem_Laplace_fractional}.
\end{proof}

\begin{lemma}[see \cite{Anhblowup}]\label{lem_scaling} 
Let $s\in (0,1)$ and $k>0$. Let $\varphi$ be a smooth function satisfying $\partial^2_x\varphi\in L^{\infty}(\mathbb{R}^n)$. For any $R>0$, let $\varphi_R$ be a function defined by
\[
\varphi_R(x):=\varphi\big(R^{-k}x\big)
\]
for all $x\in \mathbb{R}^n$. Then, $(-\Delta)^s(\varphi_R)$ satisfies the following scaling properties for all $x\in \mathbb{R}^n$:
\[
(-\Delta)^s(\varphi_R)(x)=R^{-2sk}\big((-\Delta)^{s}\varphi\big)\big(R^{-k}x\big).
\]
\end{lemma}

\end{document}